\documentclass[a4paper,10pt,reqno]{amsart}
\usepackage{amsmath}
\usepackage{cases}
\usepackage{amsfonts}
\usepackage{latexsym, amssymb, amsmath, amsthm, bbm}
\usepackage[all]{xy}
\usepackage{pgfplots}
\usepackage{enumerate}
\usepackage{mathrsfs}
\usepackage{array}
\usepackage{longtable}
\usepackage[colorlinks,linkcolor=blue,citecolor=blue]{hyperref}

\DeclareSymbolFont{EulerExtension}{U}{euex}{m}{n}
\DeclareMathSymbol{\euintop}{\mathop} {EulerExtension}{"52}
\DeclareMathSymbol{\euointop}{\mathop} {EulerExtension}{"48}

\allowdisplaybreaks[4]

\def \soc{\operatorname{soc}}
\def \dim{\operatorname{dim}}

\def \Z{\mathbb{Z}}

\def \k{\Bbbk}
\def \A{\mathcal{A}}
\def \B{\mathcal{B}}
\def \a{\alpha}
\def \b{\beta}
\newcommand{\YD}[1]{{}^{#1}_{#1}\mathcal{YD}}
\newcommand{\biproduct}{\#} % Bosonization 符号

\usepackage{mathtools}

\numberwithin{equation}{section}

\newtheorem{theorem}{Theorem}[section]
\newtheorem{lemma}[theorem]{Lemma}
\newtheorem{proposition}[theorem]{Proposition}
\newtheorem{corollary}[theorem]{Corollary}
\newtheorem{definition}[theorem]{Definition}

\newtheorem{remark}[theorem]{Remark}

\begin{document}
\title[Simple Yetter-Drinfeld modules over Generalized Liu algebras]{Simple Yetter-Drinfeld modules over Generalized Liu algebras}
\author[X. Zhen]{Xiangjun Zhen$^\dag$}
\author[G. Liu]{Gongxiang Liu}
\author[J. Yu]{Jing Yu}
\address{Department of Mathematics, Nanjing University, Nanjing 210093, China}
\email{xjzhen@smail.nju.edu.cn}
\address{Department of Mathematics, Nanjing University, Nanjing 210093, China}
\email{gxliu@nju.edu.cn}
\address{School of Mathematical Sciences, University of Science and Technology of China, Hefei 230026, China}
\email{yujing46@ustc.edu.cn}

\thanks{$^\dag$ Corresponding author}
\maketitle

\date{}

\begin{abstract}
Let $H$ be a generalized Liu algebra over an algebraically closed field $\k$ of characteristic zero. We prove that all simple Yetter-Drinfeld modules over $H$ are finite-dimensional and present an explicit classification of these modules. Moreover, we completely determine which of them admit a finite-dimensional Nichols algebra.
\end{abstract}

\maketitle
\section{Introduction}

Yetter-Drinfeld modules over a bialgebra were first introduced by Yetter \cite{Y90} in 1990, where they were originally called crossed bimodules. They received their current name in \cite{RT}. For any finite-dimensional Hopf algebra $H$ over a field $\k$, Majid \cite{Mj91} identified these modules with the modules over the Drinfeld double $D(H^{cop})$ via the category equivalences ${}^H_H \mathcal{YD} \approx {}_{H^{cop}} \mathcal{YD}^{H^{cop}} \approx {}_{D(H^{cop})}\mathcal{M}$.

Now suppose that $H$ is a Hopf algebra with a bijective antipode. Then the category ${}^H_H \mathcal{YD}$ is a braided monoidal category. A natural and important problem is to completely classify the Yetter-Drinfeld modules in ${}^H_H \mathcal{YD}$. Such a classification serves as a necessary step for further studies, most notably the classification of pointed Hopf algebras using the lifting method of Andruskiewitsch and Schneider \cite{AS98, AS10}.

As a crucial step towards understanding the entire category of Yetter-Drinfeld modules, it is natural to first investigate its simple objects. Consequently, many authors have focused on the classification of simple Yetter-Drinfeld modules. As mentioned above, for finite-dimensional Hopf algebras, this problem is usually solved by studying the representations of the Drinfeld double. For example, if $H$ is a factorizable Hopf algebra, Reshetikhin and Semenov-Tian-Shansky \cite{RS88} proved that the Drinfeld double $D(H)$ is isomorphic to a twist of $H \otimes H$. This result makes it much easier to construct Yetter-Drinfeld modules. 
Later, Schneider \cite{S01} showed that $H$ is factorizable if and only if $D(H)$ is isomorphic to such a twist. Furthermore, for a semisimple Hopf algebra $H$, it was established in \cite{B17} that any irreducible representation of its Drinfeld double $D(H)$ can be obtained as an induced representation from a certain Hopf subalgebra. More recently, Liu and Zhu \cite{LZ19} described the simple Yetter-Drinfeld modules over semisimple and cosemisimple quasi-triangular Hopf algebras. Cohen and Westreich \cite{CW24} also gave another description when the base field $\k$ is algebraically closed and of characteristic zero.

Although the general theory is well understood, finding the exact classification of simple Yetter-Drinfeld modules depends heavily on the specific Hopf algebra. Therefore, researchers often have to solve this problem case by case.

The easiest case is the group algebra $H = \k G$ of a finite group $G$ over an algebraically closed field $\k$ of characteristic zero. In this situation, the simple modules over $D(H)$ were completely classified by Dijkgraaf, Pasquier, and Roche \cite{DPR90}, and independently by Gould \cite{G93}. In addition to group algebras, many researchers have studied the Drinfeld double of bosonizations over group algebras. Specifically, let $G$ be an abelian group and $W$ be a Yetter-Drinfeld module over $\k G$. There are many results on the representations of the Drinfeld double $D(\B(W) \# \k G)$ and its slight variations. For example, specific algebras were calculated in \cite{ADP22, AAMR18, AP21, ABDF25, AJS94, ARS10}, while more general results were established in \cite{H10, HY10, L91, PV16, RS08}.

Beyond group algebras, simple Yetter-Drinfeld modules have also been classified for many other semisimple Hopf algebras. Examples include low-dimensional cases like the Kac-Paljutkin algebra \cite{HZ07, Shi_2019} and Kashina's sixteen-dimensional Hopf algebras \cite{LLK22, ZGH21, ZGH2021, ZY25}, as well as several other special cases \cite{S22, S24}.

For non-semisimple Hopf algebras, there are also many important results. In the pointed case, the Yetter-Drinfeld modules over Taft algebras \cite{C00} and generalized Taft algebras \cite{MBG21} have been completely determined. In the non-pointed case, researchers have studied several specific examples \cite{GA19, HX2018, X19, X23, X24, XH21, Zheng2024}, but a general classification is still an open problem.

Existing work on the classification of Yetter-Drinfeld modules primarily focuses on finite-dimensional Hopf algebras, where results are typically achieved by computing the representations of the Drinfeld double. For infinite-dimensional Hopf algebras, this standard method fails, and consequently, far fewer classification results are available. A notable exception is the classification of simple Yetter-Drinfeld modules over the infinite dihedral group $\mathbb{D}_{\infty}$, established by Zhang \cite{Zhang23}.

The group algebra $\k \mathbb{D}_{\infty}$ constitutes one of the five classes of affine prime regular Hopf algebras of GK-dimension one, as classified by Ding, Liu, and Wu \cite{DLW16}. Building on this framework, we recently investigated another algebra within this family—the infinite-dimensional Taft algebra—and completely classified its simple Yetter-Drinfeld modules in \cite{ZLY25}. In the present article, we continue this line of inquiry by studying yet another crucial member of this family: the generalized Liu algebras $ B(n,w,\gamma) $. Unlike their finite-dimensional counterparts, these infinite-dimensional algebras exhibit considerably richer representation-theoretic properties and present distinct structural challenges due to their non-semisimplicity. Determining the Yetter-Drinfeld modules over generalized Liu algebras represents a significant step forward in the broader program of classifying Hopf algebras of GK-dimension one (see \cite[Question 1.15]{BZ21}).

Our first main result is Theorem \ref{thm:YDmodule}:
\begin{theorem}
Let $H = B(n,w,\gamma)$.
\begin{itemize}
	\item[(1)] Every simple Yetter-Drinfeld module over $H$ is isomorphic to $V(\a,\b,x^rg^i)$ for some scalars $\a, \b \in \k^*$ and integers $r, i \in \Z$ satisfying $\a^w = \b^n$.
	\item[(2)] Two such modules $V(\a,\b,x^rg^i)$ and $V(\a',\b',x^{r'}g^{i'})$ are isomorphic if and only if
	\[
	\a = \a', \quad \b = \b', \quad \text{and} \quad x^r g^i = x^{r'} g^{i'}.
	\]
\end{itemize}
\end{theorem}
Moreover, we completely determine which of these modules admit a finite-dimensional Nichols algebra. This is the content of Theorem \ref{thm:nicholsalgebra}, which reads as follows:
\begin{theorem}
Let $H = B(n, w, \gamma)$. Then the simple Yetter-Drinfeld modules $V$ over $H$ for which the Nichols algebra $\B(V)$ is finite-dimensional are precisely those classified in Tables \ref{tab:dim1_classification}--\ref{tab:dim6_classification} following Lemma \ref{lem:finiten-nichols-classification}.
\end{theorem}
It is worth noting that $ H $ has infinitely many non-isomorphic finite-dimensional simple Yetter-Drinfeld modules, most of which cannot be realized as Yetter-Drinfeld modules over a finite-dimensional pointed Hopf algebra (see Remark \ref{rmk:nonunique_realization}).

The paper is organized as follows. In Section \ref{section2}, we recall the basic definitions and notations concerning Yetter-Drinfeld modules and Nichols algebras. We also briefly review the concepts of bosonization and 2-cocycle deformations. Furthermore, we introduce generalized Liu algebras along with their finite-dimensional counterparts. In Section \ref{section3}, we first prove that every simple Yetter-Drinfeld module over a generalized Liu algebra is finite-dimensional. We then show that each such module admits a standard basis and explicitly describe its comodule structure with respect to this basis. In Section \ref{section4}, we present the explicit construction of all simple Yetter-Drinfeld modules and completely classify their isomorphism classes. Finally, in Section \ref{section5}, we determine precisely those modules $V$ for which the associated Nichols algebra $\B(V)$ is finite-dimensional.

\section{Preliminaries}\label{section2}

Throughout this paper $\k$ denotes an algebraically closed field of characteristic 0 and all spaces are over $\k$. The tensor product over $\k$ is denoted simply by $\otimes$. The set of natural numbers including zero is denoted by $\mathbb{N}_0$, and the positive integers are denoted by $\mathbb{N}$. For any integer $m \geq 1$, let $R_m$ denote the set of primitive $m$-th roots of unity, and let $R_\infty$ denote the group of all roots of unity.

The symbol $ H $ will always denote a Hopf algebra over $ \k $ with comultiplication $ \Delta $, counit $ \varepsilon $, and antipode $ S $. We will use the Sweedler’s sigma notation \cite{Swe69} for coproduct and coaction: $ \Delta(h)=\sum h_{(1)}\otimes h_{(2)} $ for coalgebras and $ \delta(v)=\sum v_{(-1)} \otimes v_{(0)} $ for left comodules. The summation sign is often omitted when no explicit computation is involved. Denote by $ G(H)$ the set of group-like elements of $ H $. For $ g,h \in G(H)  $, the linear space of $ (g,h) $-primitives is:
\[
\mathcal{P}_{g,h} (H)= \left\{x\in H \mid \Delta(x)=x\otimes g+h\otimes x \right\}.
\]
We refer to \cite{MonS,Swe69} for the basics about Hopf algebras.

\subsection{Braided vector spaces and Yetter-Drinfeld modules}

We briefly recall the basic concept of a braided vector space (see, e.g., \cite[Definition 1.5.1]{HS20}). Let $ V $ be a vector space over $ \k $ and $ c\in \mathrm{Aut}(V \otimes V) $. The pair $ (V, c) $ is called a \emph{braided vector space} if $ c $ satisfies
\begin{equation*}
	(c\otimes \mathrm{id})(\mathrm{id}\otimes c )(c\otimes \mathrm{id}) = (\mathrm{id}\otimes c )(c\otimes \mathrm{id})(\mathrm{id}\otimes c). 
\end{equation*}
Here, the map $c$ is  called a \emph{braiding}.

Let $ I $ be an index set, and let $ (q_{i,j})_{i,j\in I} $ be a family of non-zero scalars in $ \k $. Let $ V $ be a vector space with basis $ \left\{x_i\right\}_{i\in I} $. Define a linear map $c: V \otimes V \to V \otimes V$ by
\begin{equation*}
	c(x_i \otimes x_j) = q_{i,j} x_j \otimes x_i \quad \text{for all } i,j \in I.
\end{equation*}
Then $(V,c)$ is a braided vector space (see, e.g., \cite[Remark 1.5.4]{HS20}). One says that $ (V, c) $ is a braided vector space of \emph{diagonal type}. The matrix $ (q_{i,j})_{i,j\in I} $
is called the braiding matrix of $ (V, c) $ with respect to the basis $\left\{x_i\right\}_{i\in I} $.
Braidings of diagonal type form the simplest class of braidings. Another interesting class is the braidings of \emph{triangular type}, studied by Stefan \cite{ST04}. He later showed in \cite{ST07} that such braidings come from certain Yetter-Drinfeld modules over pointed Hopf algebras with an abelian coradical.

We now recall the notion of Yetter-Drinfeld modules. Throughout this article, we work exclusively with left-left Yetter-Drinfeld modules, namely, those that are simultaneously left modules and left comodules.

\begin{definition}\emph{(}\cite[Definition 3]{AN17}\emph{)}
Let $ H $ be a Hopf algebra, a Yetter-Drinfeld module over $ H $ is a vector space provided with
\begin{itemize}
	\item [(1)] a structure of left $ H $-module $ \cdot :H\otimes V\rightarrow V $ and
	\item [(2)] a structure of left $ H $-comodule $ \delta : V\rightarrow H \otimes V  $, such that
	\item [(3)] for all $ h\in H $ and $v\in V$,the following compatibility condition holds:	
	\begin{equation}\label{compatibilitycondition}
		\delta(h\cdot v)=h_{(1)}v_{(-1)}S(h_{(3)})\otimes h_{(2)}\cdot v_{(0)}.
	\end{equation}
\end{itemize}
\end{definition}
The category of Yetter-Drinfeld modules over $H$ is denoted by ${}^H_H \mathcal{YD}$. The morphisms in this category are linear maps that preserve both the action and the coaction. If the antipode $S$ is bijective, ${}^H_H \mathcal{YD}$ becomes a braided monoidal category. For any $V, W \in {}^H_H \mathcal{YD}$, the braiding $c_{V,W} : V \otimes W \to W \otimes V$ is defined by
$$c_{V,W}(v \otimes w) = v_{(-1)} \cdot w \otimes v_{(0)}$$
and its inverse $c_{V,W}^{-1}: W \otimes V \to V \otimes W$ is given by
$$c_{V,W}^{-1}(w \otimes v) = v_{(0)} \otimes S^{-1}(v_{(-1)}) \cdot w$$
for all $v \in V$ and $w \in W$.

In particular, taking $W = V$ shows that every Yetter-Drinfeld module $V$ is naturally a braided vector space with the braiding $c_{V,V}$. However, the converse requires an extra condition: a braided vector space can be realized as a Yetter-Drinfeld module over a Hopf algebra with a bijective antipode only if its braiding is rigid (see \cite[Definition 4.2.11]{HS20}). Lyubashenko \cite{Ly86} first showed this for symmetries, and Schauenburg \cite{Sch92} later proved the general case.

In practice, the following lemma provides a  useful way to check the compatibility condition.

\begin{lemma}\emph{(}\cite[Lemma 2.5]{ZLY25}\emph{)}\label{lem:compatiable}
	Let $H$ be a Hopf algebra, and let $(V, \cdot, \delta)$ be a left $H$-module and a left $H$-comodule. Then $(V, \cdot, \delta)$ is a Yetter-Drinfeld module over $H$ if and only if the compatibility condition \eqref{compatibilitycondition} holds for a set of algebra generators of $H$ and a set of spanning vectors of $V$.
\end{lemma}

We conclude this subsection with the notion of a comatrix, which provides a matrix representation of a comodule. This concept will be essential in Section~\ref{section3}.

\begin{definition}
Let $(C, \Delta, \varepsilon)$ be a coalgebra over $\k$. A square matrix $\A=(c_{ij})_{r\times r}$ with entries in $C$ is called a \emph{comatrix} if, for all $1\leq i,j \leq r$, we have
\[
\Delta(c_{ij}) = \sum_{l=1}^r c_{il}\otimes c_{lj} \quad \text{and} \quad \varepsilon(c_{ij}) = \delta_{i, j},
\]
where $\delta_{i, j}$ denotes the Kronecker delta.
\end{definition}

\begin{lemma}\emph{(}\cite[Lemma 2.1.1]{HS20}\emph{)}\label{lem:comoduleequicodition}
	Let $(C,\Delta,\varepsilon)$ be a coalgebra over $\k$, and let $V$ be an $n$-dimensional vector space with a basis $\{v_1, \dots, v_n\}$. Consider a linear map $\delta : V \to C \otimes V$ defined by
	\[
	\delta(v_k) = \sum_{l=1}^n c_{kl} \otimes v_l \quad (1 \leq k \leq n, \, c_{kl} \in C).
	\]
	Let $\A=(c_{ij})_{n\times n}$. Then $(V,\delta)$ is a left $C$-comodule if and only if $\A$ is a comatrix.
\end{lemma}

\subsection{Nichols Algebras}

We now turn to Nichols algebras. Although there are several equivalent ways to define them, we follow the approach of Heckenberger and Schneider \cite{HS20}.

\begin{definition}\emph{(}\cite[Definition 1.3.12]{HS20}\emph{)}
	Let  $C = \bigoplus_{n \in \mathbb{N}_0} C(n)$  be an $\mathbb{N}_0$-graded coalgebra with projections $\pi_n: C \to C(n) $ for all $n \geq 0 $.  For each $n \geq 1$, define the map $\Delta_{1^n}^C$ as the composition:
	\begin{equation*}
		\Delta_{1^n}^C : C(n) \subseteq C \xrightarrow{\Delta^{n-1}} C^{\otimes n} \xrightarrow{\pi_1^{\otimes n}} C(1)^{\otimes n},
	\end{equation*}	
	where $\Delta^{0}=\operatorname{id}_C:C\to C$ and $\Delta^{n}=(\operatorname{id}_C\otimes \Delta^{n-1})\Delta :C \to C^{\otimes (n+1)} $.
\end{definition}

Now, let $H$ be a Hopf algebra with a bijective antipode, and take $V \in {}^H_H \mathcal{YD}$. The tensor algebra $T(V)$ naturally forms an $\mathbb{N}_0$-graded Hopf algebra in ${}^H_H \mathcal{YD}$. Its comultiplication is the algebra map $\Delta : T(V) \to T(V) \underline{\otimes} T(V)$ defined by $\Delta(v) = v \otimes 1 + 1 \otimes v$ for all $v \in V$. Because $T(V)$ is also a standard $\mathbb{N}_0$-graded coalgebra, we can use the construction above to define the map $\Delta_{1^n}^{T(V)}$ on it.

\begin{definition}\emph{(}\cite[Definition 7.1.13]{HS20}\emph{)}
	Let $H$ be a Hopf algebra with bijective antipode, and $V \in {}^H_H \mathcal{YD}$. Then
	\[
	\B(V) = T(V) \Big/ \bigoplus_{n \geq 2} \ker\left(\Delta_{1^n}^{T(V)}\right)
	\]
	is called the Nichols algebra of  $V$. $\mathcal{B}(V) $ is called diagonal type if
	$(V, c_{V,V})$ is of diagonal type.
\end{definition}

By construction, the Nichols algebra $\B(V) = \bigoplus_{n\geq 0} \B^n(V)$ is an $\mathbb{N}_0$-graded vector space. Its homogeneous components are explicitly given by
\[
\B^0(V) = \k 1, \quad \B^1(V) = V, \quad \text{and} \quad \B^n(V) = V^{\otimes n} \Big/ \ker\left(\Delta_{1^n}^{T(V)}\right) \text{ for } n \geq 2.
\]
In addition, $\B(V)$ becomes an $\mathbb{N}_0$-graded Hopf algebra in ${}^H_H \mathcal{YD}$ with the following properties:
\begin{itemize}
	\item $\B^0(V) \cong \k $,  $ \B^1(V)\cong V $  in  $ ^H_H \mathcal{YD} $,
	\item $\B(V)$ is generated as an algebra by $\B^1 (V)$, and
	\item $\B(V)$ is strictly graded, that is, $P(\B(V))=\B^1(V)$.
\end{itemize}
It is a well-known fact that these properties characterize the Nichols algebra (see, e.g., \cite[Theorem 5.7]{And02}).

We can also view this quotient through the braided symmetrizer maps $S_n: V^{\otimes n} \to V^{\otimes n}$ associated with $(V, c_{V,V})$ (see \cite[Definition 1.8.10]{HS20}). By \cite[Corollary 1.9.7]{HS20}, the map $\Delta_{1^n}^{T(V)}$ is exactly $S_n$.  Furthermore, one can more generally define the Nichols algebra of any braided vector space, even without an underlying Hopf algebra (see \cite[Definition 7.1.1]{HS20}).

\begin{remark}\label{rmk:exampledim1}
\normalfont
Suppose $V = \k v$ is a $1$-dimensional Yetter-Drinfeld module (or simply a braided vector space) with braiding $c(v \otimes v) = q v \otimes v$ for some $q \in \k^*$. Recall that we assume $\operatorname{char}(\k)=0$ throughout this paper. Under this assumption, it is a well-known fact that the Nichols algebra $\B(V)$ is finite-dimensional if and only if $q$ is a non-trivial root of unity (that is, a root of unity different from $1$). For more details, see, e.g., \cite[Example 1.10.1]{HS20}.
\end{remark}

\subsection{Bosonization}

Let $H$ be a Hopf algebra over $\k$. Recall that the category $\YD{H}$ is a braided monoidal category. If $R$ is a Hopf algebra in the category $\YD{H}$ (often referred to as a \textit{braided Hopf algebra}), we can construct a new ordinary Hopf algebra, denoted by $R \biproduct H$, called the \textit{bosonization} \cite{Mj94} or \textit{Radford biproduct} \cite{Ra85}  of $R$ by $H$.

\begin{definition}\emph{(}\cite[Theorem 1 and Proposition 2]{Ra85}\emph{)}
Let $R \in \YD{H}$ be a braided Hopf algebra. The bosonization $R \biproduct H$ is defined as follows:
\begin{itemize}
	\item[(1)] As a vector space, $R \biproduct H = R \otimes H$. We denote the element $r \otimes h$ by $r \biproduct h$.
	\item[(2)] The multiplication is given by the smash product:
		\begin{equation}\label{eq:bosonization-mult}
			(r \biproduct h)(s \biproduct l) = r(h_{(1)} \cdot s) \biproduct h_{(2)}l,
		\end{equation}
		for all $r, s \in R$ and $h, l \in H$. Here, the dot $\cdot$ denotes the action of $H$ on $R$.
	\item[(3)] The comultiplication is given by the smash coproduct:
		\begin{equation}\label{eq:bosonization-comult}
			\Delta(r \biproduct h) = (r^{(1)} \biproduct (r^{(2)})_{(-1)}h_{(1)}) \otimes ((r^{(2)})_{(0)} \biproduct h_{(2)}),
		\end{equation}
		where $\Delta_R(r) = r^{(1)} \otimes r^{(2)}$ denotes the braided coproduct in $R$, and $\delta(r) = r_{(-1)} \otimes r_{(0)}$ denotes the coaction of $H$ on $R$.
	\item[(4)] The unit is $1_R \biproduct 1_H$ and the counit is $\varepsilon(r \biproduct h) = \varepsilon_R(r)\varepsilon_H(h)$.
	\item[(5)]The antipode $S$ is given by the formula:
	    \begin{equation}
		   S(r \biproduct h) = (1 \biproduct S_H(r_{(-1)}h)) (S_R(r_{(0)}) \biproduct 1),
	    \end{equation}
	where $S_R$ is the antipode of $R$ in the category $\YD{H}$.
\end{itemize}

\end{definition}

The bosonization $A = R \biproduct H$ comes equipped with canonical Hopf algebra morphisms $\iota: H \to A$ given by $\iota(h) = 1_R \biproduct h$, and $\pi: A \to H$ given by $\pi(r \biproduct h) = \varepsilon_R(r)h$, satisfying $\pi \circ \iota = \operatorname{id}_H$. Conversely, a fundamental result by Radford \cite{Ra85} states that this splitting property characterizes the bosonization structure. Specifically, let $A$ and $H$ be Hopf algebras equipped with Hopf algebra morphisms $\iota: H \to A$ and $\pi: A \to H$ such that $\pi \circ \iota = \operatorname{id}_H$. Then $A \cong R \biproduct H$, where $R = A^{\text{co } \pi} = \{ a \in A \mid (\operatorname{id} \otimes \pi)\Delta(a) = a \otimes 1 \}$ is the algebra of coinvariants, which carries a natural braided Hopf algebra structure in $\YD{H}$.

\subsection{2-Cocycle Deformations}

In this subsection, we recall the procedure of twisting the multiplication of a Hopf algebra by a 2-cocycle and its consequences on the category of Yetter-Drinfeld modules. We follow the exposition in \cite[Section 3.4.3]{AN17}. 

Let $H$ be a Hopf algebra over $\k$.

\begin{definition}\emph{(}\cite[Definition 10]{AN17}\emph{)}
A linear map $\sigma: H \otimes H \to \k$ is called a \textit{unitary 2-cocycle} if it is invertible with respect to the convolution product, satisfies the cocycle condition
\begin{equation*}
	\sigma(x_{(1)}, y_{(1)}) \sigma(x_{(2)}y_{(2)}, z) = \sigma(y_{(1)}, z_{(1)}) \sigma(x, y_{(2)}z_{(2)}),
\end{equation*}
and the unitary condition $\sigma(x, 1) = \sigma(1, x) = \varepsilon(x)$ for all $x, y, z \in H$.
\end{definition}

Given a unitary 2-cocycle $\sigma$, one can define a new Hopf algebra $H_{\sigma} = (H, \cdot_{\sigma}, \Delta)$. The coalgebra structure remains unchanged, while the new multiplication is defined by
\begin{equation*}
	x \cdot_{\sigma} y = \sigma(x_{(1)}, y_{(1)}) x_{(2)}y_{(2)} \sigma^{-1}(x_{(3)}, y_{(3)}), \quad \text{for } x, y \in H.
\end{equation*}

A fundamental result relates the representation theories of $H$ and $H_{\sigma}$. The following theorem establishes an equivalence between their Yetter-Drinfeld categories (see \cite[Theorem 4]{AN17}).

\begin{theorem} \label{thm:cocycle-equivalence}
Let $\sigma: H \otimes H \to \k$ be a unitary 2-cocycle.
	\begin{itemize}
		\item[(1)] There exists an equivalence of braided categories
		\[
		\mathcal{T}_{\sigma}: \YD{H} \to \YD{H_{\sigma}}, \quad V \mapsto V_{\sigma}.
		\]
		The functor $\mathcal{T}_{\sigma}$ is the identity on the underlying vector spaces, morphisms, and coactions. It transforms the action of $H$ on $V$ to the twisted action $\cdot_{\sigma}: H_{\sigma} \otimes V_{\sigma} \to V_{\sigma}$ given by:
		\begin{equation}
			h \cdot_{\sigma} v = \sigma(h_{(1)}, v_{(-1)}) (h_{(2)} \cdot v_{(0)})_{(0)} \sigma^{-1}((h_{(2)} \cdot v_{(0)})_{(-1)}, h_{(3)}),
		\end{equation}
		for all $h \in H_{\sigma}$ and $v \in V_{\sigma}$. The monoidal structure on $\mathcal{T}_{\sigma}$ is given by the natural isomorphism
		\[
		J_{V,W}: (V \otimes W)_{\sigma} \to V_{\sigma} \otimes W_{\sigma},
		\]
		defined by the formula:
		\begin{equation}
			J_{V,W}(v \otimes w) = \sigma(v_{(-1)}, w_{(-1)}) v_{(0)} \otimes w_{(0)}, \quad \text{for } v \in V, w \in W.
		\end{equation}
		
		\item[(2)] The functor $\mathcal{T}_{\sigma}$ preserves Nichols algebras. Specifically,
		\[
		\B(V)_{\sigma} \cong \B(V_{\sigma})
		\]
		as objects in $\YD{H_{\sigma}}$. 
	\end{itemize}
\end{theorem}

Since $\mathcal{T}_{\sigma}$ preserves the dimension of Nichols algebras, $\B(V)$ is finite-dimensional if and only if $\B(V_{\sigma})$ is finite-dimensional.

\subsection{Generalized Liu algebras}\label{section2.5}

In this subsection, we briefly introduce generalized Liu algebras and similar Hopf algebras.

\begin{definition}\emph{(}\cite[Section 3.4]{BZ10}\emph{)}\label{def:generalized-Liu}
Let $ n $ and $ w $ be positive integers, and let $ \gamma $ be a primitive $ n $-th root of 1. The generalized Liu algebras, denoted by $ B(n,w,\gamma) $, is generated by $ x^{\pm 1}, g $ and $ y $, subject to the relations
\begin{equation*}
	\begin{cases}
		xx^{-1}=x^{-1}x=1, \ xg=gx, \ xy=yx, \\
		yg=\gamma gy, \\
		y^n=1-x^w=1-g^n.
	\end{cases}
\end{equation*}
The comultiplication, counit and antipode of $ B(n,w,\gamma) $ are given by 
\begin{center}
	$ \Delta(x)=x\otimes x , \quad \Delta(g)=g\otimes g , \quad \Delta(y)=y\otimes g + 1 \otimes y,   $
\end{center}
\begin{center}
	$ \varepsilon(x)=1 , \quad \varepsilon(g)=1 , \quad \varepsilon(y)=0,   $
\end{center}
and
\begin{center}
	$ S(x)=x^{-1} , \quad S(g)=g^{-1} , \quad S(y)=-yg^{-1}.   $
\end{center}
\end{definition}

\begin{remark}
\normalfont{
Let $ H := B(n,w,\gamma) $. Then $ H $  is a pointed Hopf algebra, and its group of group-like elements is
\begin{equation*}
	G(H) =\left\{ g^j x^k \mid 0 \leq j < n,\  k\in \Z \right\}.	
\end{equation*}
A linear basis of $ H $ is given by 
\begin{equation*}
	\left\{ y^i g^j x^k \mid 0 \leq i,j < n,\  k\in \Z \right\}.	
\end{equation*}}
\end{remark}

Next, we introduce another family of finite-dimensional pointed Hopf algebras, denoted by $A(d, m, \gamma, \mu)$. We demonstrate that certain simple Yetter-Drinfeld modules $V$ over $B(n,w,\gamma)$ can be realized as simple Yetter-Drinfeld modules over this new family; see Lemma \ref{lem:relization}. This realization allows us to determine whether $\B(V)$ is finite-dimensional.

\begin{definition}\label{def:finite-generalized-Liu}
Let $d, m, n$ be positive integers such that $n \mid d$. Let $\gamma$ be a primitive $ n $-th root of 1 and $\mu \in \k$. The Hopf algebra  $A(d, m, \gamma,\mu)$  is generated by $ x, g $ and $ y $, subject to the relations
\begin{equation*}
	\begin{cases}
		xg=gx, \quad xy=yx, \\
		yg=\gamma gy, \\
		g^{d}=x^{m}=1, \quad y^n=\mu (1-g^n).
	\end{cases}
\end{equation*}
The comultiplication, counit and antipode of $A(d, w, \gamma)$ are given by 
\begin{center}
	$ \Delta(x)=x\otimes x , \quad \Delta(g)=g\otimes g , \quad \Delta(y)=y\otimes g + 1 \otimes y,   $
\end{center}
\begin{center}
	$ \varepsilon(x)=1 , \quad \varepsilon(g)=1 , \quad \varepsilon(y)=0,   $
\end{center}
and
\begin{center}
	$ S(x)=x^{-1} , \quad S(g)=g^{-1} , \quad S(y)=-yg^{-1}.   $
\end{center}
\end{definition}
%若n不整除d,会迫使y=0, 此时A退化为群代数

\begin{remark}\hfil
\normalfont
\begin{itemize}
	\item[(1)] When $d=n$, the algebra $A(n, m, \gamma, 1)$ is a quotient algebra of $B(n, m, \gamma)$.
	\item[(2)] The algebra $A(d, m, \gamma, \mu)$ is a variant of the Radford algebra \cite{R75}; specifically, it is constructed by adjoining a central group-like element. In particular, when $m=1$ and $\mu=0$, it coincides with the generalized Taft algebra introduced in \cite{HCZ04}.
\end{itemize}
\end{remark}

Let $G = \langle g \rangle \times \langle x \rangle$ denote the direct product of cyclic groups of orders $d$ and $m$, generated by $g$ and $x$ respectively. Note that when $n=1$, the algebra $A(d,m,\gamma,0)$ coincides with the group algebra $\k G$. Therefore, we focus on the case $n \geq 2$. In this setting, $A(d,m,\gamma,0)$ can be realized as a bosonization over $\k G$. Specifically, let $V_0 = \k\{y\}$ denote the one-dimensional Yetter-Drinfeld module over $\k G$, with the action and coaction given by:
\begin{itemize}
	\item The action  is defined by
	\begin{equation*}
		g\cdot y=\gamma y \ ,\  x \cdot y= y ;
	\end{equation*}
	
	\item The coaction is defined by
	\begin{equation*}
		\delta(y)=g \otimes y  .
	\end{equation*}	
\end{itemize}

Note that $\B(V_0)=\k[y]/(y^n)$ is the truncated polynomial algebra. We thus obtain the following isomorphism.

\begin{lemma}\label{lem:bosination}
Let $n \geq 2$.  There exists an isomorphism of Hopf algebras
\[
\psi: \B(V_0) \biproduct \k G \xrightarrow{\cong}  A(d,m,\gamma,0),
\]
determined on the generators by $\psi(y \biproduct 1) = y g^{-1}$, $\psi(1 \biproduct g) = g^{-1}$, and $\psi(1 \biproduct x) = x$.
Explicitly, the map on the basis elements is given by:
\begin{equation}
	\psi(y^i \biproduct g^j x^k) = \gamma^{\frac{i(i-1)}{2}} \, y^i g^{-(i+j)} x^k,
\end{equation}
for all $0 \leq i < n$, $0 \leq j < d$, and $0 \leq k < m$.
\end{lemma}

The algebra $A(d, m, \gamma, 1)$ is a 2-cocycle deformation of $A(d, m, \gamma, 0)$. Following \cite[Section 4.4]{GM10}, we can explicitly describe the 2-cocycle giving rise to this deformation. Define a bilinear form $\eta: A(d, m, \gamma, 0) \otimes A(d, m, \gamma, 0) \to \k$ on the basis elements $x^i g^j y^k$ by:
\begin{equation*}
	\eta(x^i g^j y^k, x^r g^t y^s) = 
	\begin{cases} 
		-\gamma^{tk} & \text{if } k+s = n, \\
		0 & \text{otherwise},
	\end{cases}
\end{equation*}
where the indices satisfy $0\leq i,r <m$, $0\leq j,t <d$, and $0\leq k,s <n$. Then $\eta$ is a Hochschild 2-cocycle on $A(d, m, \gamma, 0)$. 

Let $\sigma$ be the exponential $\sigma = e^\eta$, with inverse $\sigma^{-1} = e^{-\eta}$. Observe that since $\eta * \eta = 0$, the series truncates, yielding:
\[
\sigma = e^\eta = \varepsilon \otimes \varepsilon + \eta \quad \text{and} \quad \sigma^{-1} = e^{-\eta} = \varepsilon \otimes \varepsilon - \eta.
\]

\begin{lemma}[\cite{GM10}, Proposition 4.2]\label{lem:2cocycle}
Let $n \geq 2$. The map $\sigma = \varepsilon \otimes \varepsilon + \eta$ is a unitary 2-cocycle on $A(d, m, \gamma, 0)$. Moreover, the twisted algebra $(A(d, m, \gamma, 0))_{\sigma}$ coincides with $A(d, m, \gamma, 1)$.    
\end{lemma}

For the remainder of this paper, we fix positive integers $n,w$,  and let $\gamma$ be a primitive $n$-th root of unity. Let $H = B(n, w, \gamma)$ and denote its group of group-like elements by $G = G(H)$.

\section{Comodule Structures over Simple Yetter-Drinfeld Modules}\label{section3}

In this section, we establish that every simple Yetter-Drinfeld module over $H$ is finite-dimensional. We then show that every such module admits a standard basis and compute its comatrix relative to this basis. The analysis of these comatrices is crucial for our classification.

\subsection{Existence of Standard Bases for Simple Yetter-Drinfeld Modules }\label{subsection31}

We first prove that every simple Yetter-Drinfeld module over $H$ is finite-dimensional.

\begin{proposition}\label{prop:simplemustfinite}
Every simple Yetter-Drinfeld module $V$ over $H$ is finite-dimensional.
\end{proposition}

\begin{proof}
Let $(V, \cdot, \delta)$ be a simple Yetter-Drinfeld module over $H$. 
	
First, we establish that $V$ is finitely generated as a module over the subalgebra $\k[x, x^{-1}]$. Define $\soc(V)$ to be the sum of all simple subcomodules of $V$. Since $H$ is pointed,
$$\soc (V)=\bigoplus\limits_{h\in G(H)} V_h,$$ 
where $ V_h=\left\{v\in V \mid \delta(v)=h\otimes v   \right\} $. Choose $h \in G(H)$ such that $V_h \neq 0$. 
Let $v$ be a non-zero vector in $V_h$. Consider the $H$-submodule $M = H \cdot v$ generated by $v$. By the compatibility condition \eqref{compatibilitycondition}, for any $a \in H$, the coaction is given by
\[ \delta(a \cdot v) = a_{(1)} v_{(-1)} S(a_{(3)}) \otimes a_{(2)} \cdot v_{(0)} = a_{(1)} h S(a_{(3)}) \otimes a_{(2)} \cdot v. \]
This shows that $M$ is a subcomodule of $V$, and thus $M$ is a Yetter-Drinfeld submodule of $V$. Since $V$ is simple as a Yetter-Drinfeld module, we must have $V = M = H \cdot v$. Consequently, $V$ is a cyclic $H$-module.
	
Next, observe the structure of $H$. The relations $y^n = 1 - x^w$ and $g^n = x^w$  indicate that $H$ is a finitely generated module over the central subalgebra $R = \k[x, x^{-1}]$.
Specifically, $H$ is spanned by the set $\{g^i y^j \mid 0 \leq i, j \le n-1\}$ over $R$. Since $V$ is a cyclic $H$-module, it is finitely generated as an $R$-module. In particular, $V$ is spanned over $R$ by $\{g^i y^j \cdot v \mid 0 \leq i, j \le n-1\}$.

Now consider the action of $x$ on $V$. Since $x$ is central in $H$ and group-like, the map $\phi_\lambda: V \to V$ defined by $u \mapsto (x - \lambda) \cdot u$ is an endomorphism of $V$ as a Yetter-Drinfeld module for any $\lambda \in \k$.
By Schur's Lemma for Yetter-Drinfeld modules, any endomorphism is either zero or invertible.
Assume that $x$ does not act as a scalar. Then for all $\lambda \in \k$, the map $x - \lambda$ is invertible. This would imply that $V$ is a divisible module over the PID $R = \k[x, x^{-1}]$ (noting that $\k$ is algebraically closed).
However, a non-zero finitely generated module over a PID cannot be divisible. 
This contradiction implies that there exists some $\a \in \k$ such that $\phi_\a$ is not invertible, and hence $\phi_\a = 0$ by Schur's Lemma.

Thus, $x \cdot u = \a u$ for all $u \in V$. It follows that $V$ is spanned over $\k$ by the set $\{g^i y^j \cdot v \mid 0 \leq i, j \leq n-1\}$. Consequently, $V$ is finite-dimensional over $\k$.
\end{proof}

Thus, it suffices to consider finite-dimensional simple Yetter-Drinfeld modules. To this end, we first introduce the following concepts.

\begin{definition}\label{def:standardelement}
Let $(V, \cdot, \delta)$ be a Yetter-Drinfeld module over $H$. A nonzero element $v \in V$ is called a \emph{standard element} if there exist a group-like element $h \in G(H)$ and scalars $\a, \b \in \k^*$ such that
\begin{equation*}
	x \cdot v = \a v, \quad g \cdot v = \b v, \quad \delta(v) = h \otimes v.
\end{equation*}
In this case, we say $v$ is a standard element of type $(\a, \b, h)$.
\end{definition}

Since $x^w = g^n$, if $v$ is a standard element of type $(\a, \b, h)$, we must have $\a^w = \b^n$. 
The following lemma ensures the existence of standard elements.

\begin{lemma}\label{lem:existenceofse}
Every finite-dimensional Yetter-Drinfeld module over $H$ contains a standard element.
\end{lemma}
\begin{proof}
Let $(V, \cdot, \delta)$ be a finite-dimensional Yetter-Drinfeld module over $H$. Define $\soc(V)$ to be the sum of all simple subcomodules of $V$. Since $H$ is pointed,
$$\soc (V)=\bigoplus\limits_{h\in G(H)} V_h,$$ 
where $ V_h=\left\{v\in V \mid \delta(v)=h\otimes v   \right\} $. Choose $h \in G(H)$ such that $V_h \neq 0$. Then $V_h$ is finite-dimensional because $V$ is. By the compatibility condition \eqref{compatibilitycondition}, the subspace $V_h$ is invariant under the action of both $x$ and $g$. Since $xg = gx$ and $\k$ is algebraically closed, the linear maps $x$ and $g$ have a common eigenvector $v$ in $V_h$. This $v$ is therefore a standard element of $V$.
\end{proof}

Let $V$ be a Yetter-Drinfeld module over $H$. For any standard element $v \in V$ and integer $k \geq 0$, define $V(v,k)$ to be the linear span of $\{v, y \cdot v, \dots, y^k \cdot v\}$. We have the following immediate result.

\begin{proposition}\label{pro:spaniscomodule}
Let $V$ be a  Yetter-Drinfeld module over $H$. Let $v \in V$ be a standard element of type $  (\a,\b,x^rg^i)  $, where $ \a,\b\in k^* $ and $r,i \in \Z$. Then for any $k \geq 0$:
\begin{itemize}
	\item [(1)] $x\cdot (y^k \cdot v)=\a(y^k \cdot v), \ g\cdot (y^k \cdot v) =\b\gamma^{-k} (y^k \cdot v) $;
	\item [(2)] $V(v,k)$ is a subcomodule of $V$;
	\item [(3)] If $y^{k+1} \cdot v \in V(v,k)$, then $V(v,k)$ is a Yetter-Drinfeld submodule of $V$.
\end{itemize}
\end{proposition}

\begin{proof}
(1) is straightforward.

For (2), we proceed by induction on $k$. The case $k=0$ holds trivially. Assume $V(v,k)$ is a subcomodule. Then $\delta(y^k \cdot v) = \sum_{l=0}^k c_{kl} \otimes (y^l \cdot v)$ for some $c_{kl} \in H$. By the compatibility condition \eqref{compatibilitycondition}, 	
\begin{eqnarray*}
	& &\delta (y^{k+1} \cdot v) \\
	&=&\delta (y\cdot (y^{k} \cdot v)) \\
	&=&\sum\limits_{l=0}^k \left(c_{kl}S(y)\otimes 1\cdot (y^l \cdot v)+ c_{kl}S(g)\otimes y\cdot (y^l \cdot v)+yc_{kl}S(g)\otimes g\cdot (y^l \cdot v) \right) \\
	&=&\sum\limits_{l=0}^{k+1} c_{k+1,l} \otimes y^l \cdot v,
\end{eqnarray*}
where
\begin{equation}\label{coefficientckl}
	c_{k+1,l}=\begin{cases}
		c_{k0}S(y)+\b yc_{k0}S(g),     & l=0, \\[5pt]
		c_{kl}S(y)+\b\gamma^{-l}yc_{kl}S(g) +    c_{k,l-1}S(g),     \qquad    &     0<l<k+1,  \\[5pt]
		c_{kk} S(g),           &     l=k+1.
	\end{cases}
\end{equation}
Thus $\delta(y^{k+1} \cdot v) \in H \otimes V(v,k+1)$. By the induction hypothesis, $V(v,k+1)$ is a subcomodule.

For (3), note that by (1), $V(v,k)$ is a $\k G$-submodule. If $y^{k+1} \cdot v \in V(v,k)$, then $V(v,k)$ is $y$-stable, hence an $H$-submodule. Combined with (2), it is a Yetter-Drinfeld submodule. 
\end{proof}

\begin{lemma}\label{lem:standbasis}
Let $V$ be a simple Yetter-Drinfeld module over $H$ with $\dim_{\k} V = p+1$ for some  $p \geq 0$. For any standard element $v \in V$, the set $\{v, y\cdot v, \dots, y^p \cdot v\}$ forms a basis of $V$. 
\end{lemma}

\begin{proof}
Let $k$ be the largest integer such that $\{v, y\cdot v, \dots, y^k \cdot v\}$ is linearly independent.  Then $y^{k+1} \cdot v \in V(v,k)$. By Proposition \ref{pro:spaniscomodule}(3), $V(v,k)$ is a Yetter-Drinfeld submodule. As $V$ is simple, $V = V(v,k)$. Therefore, the set $\{v, y \cdot v, \dots, y^k \cdot v\}$ is a basis of $V$, implying $\dim_{\k} V = k+1$. Since $\dim_{\k} V = p+1$, we obtain $k = p$, which concludes the proof.
\end{proof}

The basis presented in Lemma~\ref{lem:standbasis} will be referred to as a \emph{standard basis}; using it, we derive the results below.

\begin{lemma}\label{lem:Vdimension}
Let $V$ be a simple Yetter-Drinfeld module over $H$ with $\dim_{\k} V = p+1$ for some $p \geq 0$. Let $v \in V$ be a standard element of type $(\a,\b,h)$, where $\a, \b \in \k^*$ and $h \in G(H)$. Then:
\begin{itemize}
	\item[(1)] $p+1 \leq n$;
	\item[(2)] If $p+1 < n$, then $y^{p+1} \cdot v = 0$;
	\item[(3)] $y^{p+1} \cdot v = 0$ if and only if $\a^w = \b^n = 1$;
	\item[(4)] $y^{p+1} \cdot v \neq 0$ if and only if $\a^w = \b^n \neq 1$.
\end{itemize}
\end{lemma}

\begin{proof}
(1) By Lemma~\ref{lem:standbasis}, the set $\{v, y \cdot v, \dots, y^p \cdot v\}$ forms a basis of $V$. Moreover, we have
\[
y^n \cdot v = (1 - g^n) \cdot v = (1 - \b^n) v.
\]
This forces $p < n$, which proves statement (1).

(2) Assume that $y^{p+1} \cdot v \neq 0$. By Proposition~\ref{pro:spaniscomodule}(1), for each $k$ with $0 \leq k \leq p+1$, we have
\[
g \cdot (y^k \cdot v) = \b \gamma^{-k} (y^k \cdot v).
\]
Thus, the vectors $v, y \cdot v, \dots, y^{p+1} \cdot v$ are eigenvectors of $g$ corresponding to distinct eigenvalues $\b\gamma^{-k}$, and hence they are linearly independent. This contradicts the fact that $\{v, y \cdot v, \dots, y^p \cdot v\}$ is a basis of $V$. Therefore, we must have $y^{p+1} \cdot v = 0$.

Since (3) and (4) are equivalent, it suffices to prove (3). First, assume that $y^{p+1} \cdot v = 0$. Since $p+1 \leq n$ (by part (1)), it follows that
\[
0 = y^{n-(p+1)} \cdot (y^{p+1} \cdot v) = y^n \cdot v = (1 - \b^n)v.
\]
This implies $\b^n = 1$. Since $\a^w = \b^n$, it follows that $\a^w = 1$.

Conversely, assume that $\a^w = \b^n = 1$. By (1), we know $p+1 \leq n$.
If $p+1 < n$, then by (2) we have $y^{p+1} \cdot v = 0$.
If $p+1 = n$, then
\[
y^{p+1} \cdot v = y^n \cdot v = (1 - g^n) \cdot v = (1 - \b^n)v = 0.
\]
In both cases, $y^{p+1} \cdot v = 0$.
\end{proof}

The action and coaction of a simple Yetter-Drinfeld module admit a simple form with respect to a standard basis.
Let  $ V $  be a simple Yetter-Drinfeld module over $H$ of dimension  $ p+1 $  for some  $ p \geq 0 $. Fix a standard element  $ v \in V $ of type $  (\a,\b,x^rg^i)  $, where $ \a,\b\in k^* $ and $r,i \in \Z$. Then the set $\{v, y\cdot v, \dots, y^p \cdot v\}$ forms a standard basis of $  V  $. Denoting  $ v_k = y^k \cdot v $  for  $ 0 \leq k \leq p $, the action and coaction on this basis are given by:

\begin{itemize}
	\item The action of  $ x $  is given for $  0 \leq k \leq p $ by
	\begin{equation*}
		x\cdot v_k=\a v_k.
	\end{equation*}
	\item The action of  $ g $  is given for $  0 \leq k \leq p $ by
	\begin{equation*}
		g\cdot v_k=\b\gamma^{-k} v_k.
	\end{equation*}
	\item The action of  $ y $  is given for $  0 \leq k \leq p $ by
	\begin{equation}\label{standbasisactioncoaction}
		y\cdot v_k=
		\begin{cases}
			v_{k+1} , & 0\leq k<p  ,\\[5pt]
			\sum_{l=0}^{p} a_l v_l    ,   & k=p  ,
		\end{cases}
	\end{equation}
    where $ a_l \in \k $ for $ 0\leq l \leq p $.
	\item The coaction  is given for $  0 \leq k \leq p $ by
	\begin{equation*}
		\delta(v_k) =\sum\limits_{l=0}^{k} c_{k,l} \otimes v_l,
	\end{equation*}
    where the coefficients  $ c_{k,l} $  are defined recursively by Equation~\eqref{coefficientckl}, with initial condition  $ c_{0,0} = x^r g^i $.
\end{itemize}

\subsection{The Comatrix Relative to a Standard Basis}

The preceding subsection concluded with a description of the action and coaction on a finite-dimensional simple Yetter-Drinfeld module relative to a standard basis. We now examine the converse problem: given arbitrary parameters $ \a,\b \in \k^* $, $ r,i\in \Z $, $ p\geq 0 $, and scalars $ a_l \in \k $ ($ 0\leq l \leq p $), let $ V $ be a vector space of dimension $ p+1 $ with basis $\{v_0, v_1, \dots, v_p\}$, endowed with the action and coaction defined in \eqref{standbasisactioncoaction}.  A natural question is whether $ V $ thereby becomes a Yetter-Drinfeld module over $ H $.

Generally, the answer is negative. For example, when $p+1 < n$, it follows that $a_l = 0$ for all $l$. Moreover, the parameters $p$ and $a_l$ must be selected to also fulfill the compatibility conditions \eqref{compatibilitycondition}.

Although $ V $ may not be a full Yetter-Drinfeld module, we will prove that it always admits an $ H $-comodule structure. Our strategy is to compute the coefficient matrix of the coaction explicitly and verify that it is a comatrix. The entries of this matrix are defined recursively by Equation \eqref{coefficientckl}, as we now formalize.

\begin{definition}\label{def:ckl}
Let $r,i \in \mathbb{Z}$ and $ \b\in \k^* $, define coefficients $c_\b^{r,i}(k,l) \in H$ recursively as follows: set $c_\b^{r,i}(0,0) := x^r g^i$ and for $k \geq 0$, $0 \leq l \leq k+1$:
\begin{equation*}
	c_\b^{r,i}(k+1,l):=\begin{cases}
		c_\b^{r,i}(k,0)S(y)+\b yc_\b^{r,i}(k,0)S(g),     & l=0 ,\\[5pt]
		c_\b^{r,i}(k,l)S(y)+\b \gamma^{-l}y c_\b^{r,i}(k,l)S(g) +    c_\b^{r,i}(k,l-1)S(g),     \qquad    &     0<l<k+1,  \\[5pt]
		c_\b^{r,i}(k,k) S(g)  ,         &     l=k+1.
	\end{cases}
\end{equation*} 
For $p \geq 0$, define the $(p+1) \times (p+1)$ lower triangular matrix $\mathcal{A}_\b^{r,i}(p)$ with entries:
\[
\big( \mathcal{A}_\b^{r,i}(p) \big)_{k,l} = 
\begin{cases} 
	c_\b^{r,i}(k,l) , & 0 \leq l \leq k \leq p , \\
	0 , & \text{otherwise}.
\end{cases}
\]
Explicitly:
\[
\mathcal{A}_\b^{r,i}(p) = 
\begin{pmatrix}
	c_\b^{r,i}(0,0) & 0 & \cdots & 0 \\
	c_\b^{r,i}(1,0) & c_\b^{r,i}(1,1) & \cdots & 0 \\
	\vdots & \vdots & \ddots & \vdots \\
	c_\b^{r,i}(p,0) & c_\b^{r,i}(p,1) & \cdots & c_\b^{r,i}(p,p)
\end{pmatrix}.
\]
\end{definition}
The following properties are immediate from the definitions.

\begin{remark}\label{rmk:ckl}
\normalfont 
Let $ r,i,t,s \in \Z $ and $ \a,\b \in \k^* $.
\begin{itemize}
\item[(1)]  If  $ x^r g^i=x^t g^s  $, then $ \A_{\b}^{r,i}(p) =\A_{\b}^{t,s}(p)  $ for all $ p\geq 0 $. In other words, the coefficient $c_\b^{r,i}(k,l)$ depends only on the group-like element $x^r g^i$ and the scalar $\b$.

\item[(2)] For any Yetter-Drinfeld module $V$ over $H$ and $v \in V$ is a standard element of type $ (\a,\b,x^rg^i) $, then 
\begin{equation}\label{deltax^kv}
	\delta(y^k\cdot v)=\sum\limits_{l=0}^k c_\b^{r,i}(k,l) \otimes y^l \cdot v 
\end{equation}
holds for all $k \geq 0$.
\end{itemize}
\end{remark}

To express the entries $c_\b^{r,i}(k,l)$ explicitly, we define the following coefficient families over $\k$. Note that they depend only on the parameters $i$ and $\b$.

\begin{definition}
Let $i \in \mathbb{Z}$ and $ \b \in \k^* $, define the following coefficient families in $\k$:
	\begin{itemize}
		\item[(1)] For $k \geq 0$ and $0 \leq l \leq k$:
		 \begin{equation*}
		  R_\b^i(k,l):=\begin{cases}
			\b\gamma^{-l}-\gamma^{k-1-i},  \qquad  & 0\leq l <k, \\[5pt]
			1 ,          &     l=k .
		             \end{cases}
	     \end{equation*} 
		
		\item[(2)] Set $\lambda_\b^i(0,0) := 1$, and for $k \geq 1$, $0 \leq l \leq k$:
        \begin{equation*}
        	\lambda_\b^i(k,l):=\begin{cases}
        		R_\b^i(k,0)\lambda_\b^i(k-1,0) , & l=0, \\[5pt]
        		R_\b^i(k,l)\lambda_\b^i(k-1,l)+\lambda_\b^i(k-1,l-1) , \qquad  & 0< l <k , \\[5pt]
        		1  ,         &     l=k .
        	                  \end{cases}
        \end{equation*} 
	\end{itemize}
\end{definition}

\begin{remark}\label{rmk:lambdak0}\hfil
\normalfont
\begin{itemize}
	\item [(1)] For all $k \geq 1$, we have $\lambda_\b^i(k,0) = \prod_{l=1}^{k} (\b-\gamma^{l-1-i})$. In particular, $\lambda_\b^i(n,0) = \b^n - 1$.
	\item [(2)] If $\b=\gamma^j$ for some $j\in \Z$, then $\lambda_\b^{i}(k,l)$ (respectively, $R_\b^i(k,l)$) coincides with $\lambda_j^{i}(k,l)$ (respectively, $R_j^i(k,l)$) defined for the infinite-dimensional Taft algebra $H(n,1,\gamma)$ (see \cite[Definition 3.8]{ZLY25}).
\end{itemize}
\end{remark}

Next, we establish the relationship between  $c_\b^{r,i}(k,l) $  and $\lambda_\b^i(k,l) $.

\begin{proposition}\label{pro:cijformula}
Let  $r,i \in \mathbb{Z}$ and $\b \in k^*$. Then
\[
c_\b^{r,i}(k,l) = \lambda_\b^i(k,l)y^{k-l}x^rg^{i-k}
\]
for all $0 \leq l \leq k$.
\end{proposition}

\begin{proof}
We proceed by induction on $k$. The base case $k=0$ holds by definition. For the inductive step, assume that
\[
c_\b^{r,i}(k,l) = \lambda_\b^i(k,l)y^{k-l}x^rg^{i-k}
\]
holds for all $0 \leq l \leq k$.

For the case $l = k+1$, Definition \ref{def:ckl} immediately yields $c_\b^{r,i}(k+1,k+1) = x^rg^{i-(k+1)}$.
For $0 < l < k+1$, we have
\begin{eqnarray*}
	& & c_\b^{r,i}(k+1,l) \\
	&=& c_\b^{r,i}(k,l)S(y)+\b \gamma^{-l}y c_\b^{r,i}(k,l)S(g) +    c_\b^{r,i}(k,l-1)S(g)  \\[5pt]
	&=&\left(\lambda_\b^i(k,l)y^{k-l}x^rg^{i-k}\right) S(y)+ \b\gamma^{-l} y\left(\lambda_\b^i(k,l)y^{k-l} x^r g^{i-k}\right)S(g)+  \\[5pt]
	& & \left(\lambda_\b^i(k,l-1)y^{k-l+1}x^r g^{i-k}\right)S(g) \\[5pt]
	&=&\left(-\gamma^{k-i} \lambda_\b^i(k,l)+ \b\gamma^{-l}\lambda_\b^i(k,l) + \lambda_\b^i(k,l-1) \right)y^{k+1-l} x^r g^{i-(k+1)} \\[5pt]
	&=&\left(R_\b^i(k+1,l) \lambda_\b^i(k,l) + \lambda_\b^i(k,l-1) \right)y^{k+1-l} x^r g^{i-(k+1)} \\[5pt]
	&=&\lambda_\b^i(k+1,l)y^{k+1-l} x^r g^{i-(k+1)}. 
\end{eqnarray*}
Similarly, the case $c_\b^{r,i}(k+1,0) = \lambda_\b^i(k+1,0)y^{k+1} x^r g^{i-(k+1)}$ follows by a similar calculation. This completes the inductive step and the proof.
\end{proof}

Next, we verify that $\mathcal{A}_\b^{r,i}(m)$ is a comatrix. To begin, we provide the explicit formula for $\lambda_\b^i(k,l)$, which is analogous to Proposition 3.11 in \cite{ZLY25}. We adopt the convention that $\prod_{l=p+1}^{k} R_\b^{i}(l,0) = 1$ when $p = k$ (an empty product).

\begin{proposition}\label{prop:explictlambda}
For all  $i\in \Z$, $ \b \in \k^* $ and $0\leq p\leq k  $, 
\begin{equation}\label{explictlambda}
\lambda_\b^i (k,p) = \binom{k}{p}_{\gamma} \gamma^{-(k-p)p} \prod\limits_{l=p+1}^{k} R_\b^{i}(l,0).
\end{equation}
\end{proposition}

\begin{proof}
We proceed by induction on $k$. The base case $k=0$ holds by definition. 
Assume that \eqref{explictlambda} holds for a fixed $k \geq 0$ and all $0 \leq p \leq k$. Consider the case $k+1$. For $0 \leq p \leq k+1$:
\begin{itemize}
	\item When $p=0$ or $p=k+1$, \eqref{explictlambda} holds by definition.
		
	\item For $0 < p < k+1$, the recursive definition gives:
		
\end{itemize}	
\begin{eqnarray*}
	& &\lambda_\b^{i}(k+1,p)   \\[5pt]
	&=&R_\b^{i}(k+1,p)\lambda_\b^{i}(k,p)+\lambda_\b^{i}(k,p-1)   \\[5pt]
	&=&(\b\gamma^{-p}-\gamma^{k-i}) \binom{k}{p}_{\gamma} \gamma^{-(k-p)p} \prod\limits_{l=p+1}^{k} R_\b^{i}(l,0) +\binom{k}{p-1}_{\gamma} \gamma^{-(k-p+1)(p-1)} \prod\limits_{l=p}^{k}R_\b^{i}(l,0)  \\[5pt]
	&=&\gamma^{-(k+1-p)p} \left(\prod\limits_{l=p+1}^{k} R_\b^{i}(l,0)  \right)  \left[ (\b\gamma^{-p}-\gamma^{k-i}) \gamma^p \binom{k}{p}_{\gamma}   +  \right.  \\[5pt]
	& &  \left.  (\b-\gamma^{p-1-i}) \gamma^{k+1-p} \binom{k}{p-1}_{\gamma}  \right]
		\\[5pt]
	&=&\gamma^{-(k+1-p)p} \left(\prod\limits_{l=p+1}^{k} R_\b^{i}(l,0)  \right) \left[\b \left(\binom{k}{p}_{\gamma}+ \gamma^{k+1-p}\binom{k}{p-1}_{\gamma} \right) - \right.    \\[5pt]
	& &  \left.  \gamma^{k-i} \left(\gamma^{p}\binom{k}{p}_{\gamma}+ \binom{k}{p-1}_{\gamma}  \right)   \right]
		\\[5pt]
	&=& \gamma^{-(k+1-p)p} \left(\prod\limits_{l=p+1}^{k} R_\b^{i}(l,0)  \right) \left(\b  \binom{k+1}{p}_{\gamma}- \gamma^{k-i} \binom{k+1}{p}_{\gamma} \right)
		\\[5pt]
	&=& \binom{k+1}{p}_{\gamma} \gamma^{-(k+1-p)p} \prod\limits_{l=p+1}^{k+1} R_\b^{i}(l,0).
\end{eqnarray*}
This completes the induction.	
\end{proof}

We now establish that $\A_\b^{r,i}(m)$ is a comatrix.
\begin{corollary}\label{cor:Aijisacomatrix}
For all $r,i \in \Z$, $\b \in \k^*$ and $m \geq 0$, $\A_\b^{r,i}(m)$ is a comatrix.
\end{corollary}

\begin{proof}
We set $c_\b^{r,i}(k,l) = 0$ for $l > k$. For any $0 \leq l,k \leq m$, we clearly have
\[
\varepsilon ( c_\b^{r,i}(k,l)) = \delta_{k,l}.
\]
It remains to show that
\[
\Delta(c_\b^{r,i}(k,l)) = \sum_{p=0}^m c_\b^{r,i}(k,p) \otimes c_\b^{r,i}(p,l).
\]
The case $l > k$ is trivial. Thus, it suffices to consider $0 \leq l \leq k$.
On the one hand, computing the left-hand side yields:
\begin{eqnarray*}
	& &\Delta(c_\b^{r,i}(k,l))  \\[5pt]
	&=&\lambda_\b^i(k,l) \Delta(y)^{k-l} \Delta(x)^{r} \Delta(g)^{i-k} \\[5pt]
	&=&\lambda_\b^i(k,l) (y\otimes g +1 \otimes y)^{k-l} (x\otimes x)^r (g\otimes g)^{i-k}\\[5pt]
	&=&\lambda_\b^i(k,l) \left(\sum\limits_{p=0}^{k-l} \binom{k-l}{p}_{\gamma}(y\otimes g)^{k-l-p} (1\otimes y)^p \right) (x^r \otimes x^r)  (g^{i-k}\otimes g^{i-k})\\[5pt]
	&=&\sum\limits_{p=0}^{k-l} \binom{k-l}{p}_{\gamma} \lambda_\b^i(k,l)  \gamma^{-(k-(p+l))p} \left(y^{k-(p+l)}x^rg^{i-k} \otimes y^p x^r g^{i-(p+l)} \right).  
\end{eqnarray*}	
On the other hand, for the right-hand side, we have:
\begin{eqnarray*}
\sum_{p=0}^m c_\b^{r,i}(k,p) \otimes c_\b^{r,i}(p,l)
&=& \sum_{p=l}^k c_\b^{r,i}(k,p) \otimes c_\b^{r,i}(p,l) \\
&=& \sum_{p=l}^k \lambda_\b^i(k,p) \lambda_\b^i(p,l) y^{k-p}x^rg^{i-k} \otimes y^{p-l}x^r g^{i-p}.
\end{eqnarray*}
By comparing the coefficients of the term $y^{k-(p+l)} x^r g^{i-k} \otimes y^p x^r g^{i-(p+l)}$ for any $0 \leq p \leq k-l$, it suffices to show that
\[
\lambda_\b^i(k,p+l)\lambda_\b^i(p+l,l) = \binom{k-l}{p}_{\gamma} \lambda_\b^i(k,l) \gamma^{-(k-(p+l))p}.
\]
This equality follows directly from Proposition \ref{prop:explictlambda}.
\end{proof}

\subsection{Further Analysis of the Comatrix}\label{subsection3}

In this subsection, we provide a detailed analysis of the comatrices defined in the preceding subsection, thereby deriving key properties of finite-dimensional simple Yetter-Drinfeld modules over $H$. These properties are essential to our classification.

Let $V$ be a simple Yetter-Drinfeld module over $H$ with $\dim_{\k} V = p+1$ for some $p \geq 0$. Let $v \in V$ be a standard element of type $(\a,\b,x^rg^i)$, where $\a,\b \in \k^*$ and $r,i \in \Z$. By Lemma \ref{lem:standbasis}, the set $\{v, y \cdot v, \dots, y^p \cdot v\}$ is a basis of $V$. The comodule structure of $V$ with respect to this basis is determined by the matrix $\A_\b^{r,i}(p)$.

We now turn our attention to the matrix $\A_\b^{r,i}(p+1)$. We consider a matrix of one order higher for several reasons. First, in Lemma \ref{lem:dimensionp}, it enables us to prove that the dimension of $V$ is determined by $\b$ and $i$. Second, in Lemma \ref{lem:yanzhengYDmo}, verifying the compatibility condition for the constructed Yetter-Drinfeld module requires information regarding the entries in the $(p+2)$-th row.

First, consider the case $\b^n \neq 1$. By Lemma \ref{lem:Vdimension}(2) and (3), we must have $p=n-1$. Then, by Proposition \ref{prop:explictlambda}, $c_\b^{r,i}(k,0) \neq 0$ for all $0 \leq k \leq n$, and $c_\b^{r,i}(n,l) = 0$ for all $1 \leq l \leq n-1$. Thus, we obtain the following result.

\begin{lemma}\label{lem:Abnotroot}
Let $r, i \in \Z$ and $\b \in \k^*$ such that $\b^n \neq 1$. Then the entries in the first column of $\mathcal{A}_\b^{r,i}(n)$ are all non-zero, and the entries in the last row are zero except for the first and last ones. Explicitly,
\[
\mathcal{A}_\b^{r,i}(n) = 
\begin{pmatrix}
	c_\b^{r,i}(0,0) & 0 & \cdots & 0 & 0 \\
	c_\b^{r,i}(1,0) & c_\b^{r,i}(1,1) & \cdots & 0 & 0 \\
	\vdots & \vdots & \ddots & \vdots & \vdots \\
	\vdots & \vdots & & \ddots & 0 \\
	c_\b^{r,i}(n,0) & 0 & \cdots & 0 & c_\b^{r,i}(n,n)
\end{pmatrix}.
\]
\end{lemma}

Next, consider the case $\b^n = 1$. We immediately obtain the following result.

\begin{lemma}\label{lem:lastrow0}
Let $V$ be a simple Yetter-Drinfeld module over $H$ with $\dim_{\k} V = p+1$ for some $p \geq 0$. Let $v \in V$ be a standard element of type $(\a,\b,x^rg^i)$, where $\a,\b \in \k^*$ and $r,i \in \Z$. If $\b^n = 1$, then $c_\b^{r,i}(p+1,l) = 0$ for all $0 \leq l \leq p$. Explicitly,
\[
\mathcal{A}_\b^{r,i}(p+1) = 
\begin{pmatrix}
	c_\b^{r,i}(0,0) & 0 & \cdots & 0 & 0 \\
	c_\b^{r,i}(1,0) & c_\b^{r,i}(1,1) & \cdots & 0 & 0 \\
	\vdots & \vdots & \ddots & \vdots & \vdots \\
	\vdots & \vdots & & \ddots & 0 \\
	0 & 0 & \cdots & 0 & c_\b^{r,i}(p+1,p+1)
\end{pmatrix}.
\]
\end{lemma}

\begin{proof}
By Lemma \ref{lem:Vdimension}(3), we have $y^{p+1} \cdot v = 0$. Then, by Remark \ref{rmk:ckl}(2), we obtain
\[
0 = \delta(y^{p+1}\cdot v) = \sum_{l=0}^{p+1} c_\b^{r,i}(p+1,l) \otimes y^l \cdot v.
\]
Since $y^{p+1} \cdot v = 0$, the term corresponding to $l=p+1$ vanishes. Thus, the equation simplifies to
\[
\sum_{l=0}^{p} c_\b^{r,i}(p+1,l) \otimes y^l \cdot v = 0.
\]
By Lemma \ref{lem:standbasis}, the set $\{v, y \cdot v, \dots, y^p \cdot v\}$ is a basis of $V$. Therefore, the linear independence of these basis elements implies that $c_\b^{r,i}(p+1,l) = 0$ for all $0 \leq l \leq p$.
\end{proof}

We proceed to show that $c_\b^{r,i}(k,l) \neq 0$ for all $0 \leq l \leq k \leq p$. To this end, we first determine the value of $p$. Since $c_\b^{r,i}(p+1,0)=0$, we introduce the following definition. We will demonstrate that $p$ coincides with the integer $m$ defined below, and that $m$ is uniquely determined by $\b$ and $i$.

\begin{definition}
For any $r, i \in \Z$ and $\b \in \k^*$ such that $\b^n=1$, we define $\overline{\A_{\b}^{r,i}} := \A_{\b}^{r,i}(m)$, where $m \geq 0$ is the smallest integer such that $c_{\b}^{r,i}(m+1, 0) = 0$.
\end{definition}

The following properties are immediate from the definitions.

\begin{remark}\label{rmk:Atijblock}\hfill
\normalfont
\begin{itemize}
	\item[(1)] Recall that $c_{\b}^{r,i}(m+1,0) = 0$ if and only if $\lambda_{\b}^i(m+1,0) = 0$. Consequently, $m$ is the smallest integer such that $\lambda_{\b}^i(m+1,0) = 0$, and equivalently, the smallest integer such that $R_{\b}^i(m+1,0) = 0$.
	\item[(2)] Proposition \ref{prop:explictlambda} implies that $\lambda_{\b}^i(k,l) \neq 0$ for all $0 \leq l \leq k \leq m$. This means that all entries in the lower triangular part (including the diagonal) of $\overline{\A_{\b}^{r,i}}$ are non-zero.
	\item[(3)] Proposition \ref{prop:explictlambda} implies that $\lambda_{\b}^i(k,l) = 0$ for all $0 \leq l \leq m < k$. Consequently, the entries in the last row of $\A_{\b}^{r,i}(m+1)$ are all zero, except for the last one. 
\end{itemize}
\end{remark}

Motivated by Remark \ref{rmk:Atijblock}(3), we observe that $\A_{\b}^{r,i}(m+1)$ shares a similar structure with $\A_{\b}^{r,i}(p+1)$ regarding the last row. We now proceed to show that $m=p$. To establish this, we first define a function $\phi$ that allows us to give an explicit expression for $m$. For any $i \in \mathbb{Z}$, let $\phi(i)$ denote the unique integer in $\{1, 2, \dots, n\}$ satisfying
\[
i \equiv \phi(i) \pmod{n}.
\]

\begin{lemma}\label{lem:dimensionp}
Let $V$ be a simple Yetter-Drinfeld module over $H$ with $\dim_{\k} V = p+1$ for some $p \geq 0$. Let $v \in V$ be a standard element of type $(\a,\b,x^rg^i)$, where $\a,\b \in \k^*$ and $r,i \in \Z$. If $\b = \gamma^j$ for some $j \in \Z$, then
\[
p = m = n - \phi(-i-j),
\]
where $m \geq 0$ is the smallest integer such that $c_{\b}^{r,i}(m+1, 0) = 0$.
\end{lemma}

\begin{proof}
For any $k \geq 1$, we have
\[
R_\b^{i}(k,0) = \gamma^j - \gamma^{k-1-i} = \gamma^j(1 - \gamma^{k-1-i-j}).
\]
It follows from Remark \ref{rmk:Atijblock}(1) that $m = n - \phi(-i-j)$. It remains to show that $p = m$.
	
First, suppose for the sake of contradiction that $p < m$. Then Remark~\ref{rmk:Atijblock}(2) implies $c_\b^{r,i}(p+1, 0) \neq 0$, which contradicts Lemma \ref{lem:lastrow0}. Hence, we must have $p \geq m$.
	
Conversely, suppose that $p > m$. By Remark~\ref{rmk:Atijblock}(3), the coefficients $c_\b^{r,i}(m+1, l)$ vanish for all $0 \leq l \leq m$. Consequently,
\[
\delta(y^{m+1} \cdot v) = c_\b^{r,i}(m+1,m+1) \otimes y^{m+1} \cdot v = x^rg^{i-(m+1)} \otimes y^{m+1} \cdot v.
\]
Together with Proposition \ref{pro:spaniscomodule}(1), this implies that $y^{m+1} \cdot v$ is a standard element of $V$. Since $p > m$, $y^{m+1} \cdot v \neq 0$. Thus, the subspace $W = \operatorname{span}\{y^{m+1} \cdot v, \dots, y^p \cdot v\}$ is a non-zero proper Yetter-Drinfeld submodule of $V$. This contradicts the simplicity of $V$. Therefore, we conclude that $p = m$.
\end{proof}

We summarize the results obtained in Lemma \ref{lem:Abnotroot}, Lemma \ref{lem:lastrow0}, Remark \ref{rmk:Atijblock}, and Lemma \ref{lem:dimensionp} into the following proposition.

\begin{proposition}\label{pro:summary}
Let $V$ be a simple Yetter-Drinfeld module over $H$ with $\dim_{\k} V = p+1$ for some $p \geq 0$. Let $v \in V$ be a standard element of type $(\a,\b,x^rg^i)$, where $\a,\b \in \k^*$ and $r,i \in \Z$.
	\begin{itemize}
		\item[(1)] If $\b^n \neq 1$, then $p = n-1$. In this case, the entries in the first column of $\mathcal{A}_\b^{r,i}(n)$ are all non-zero, and the entries in the last row are zero except for the first and last ones. Explicitly,
		\[
		\mathcal{A}_\b^{r,i}(n) = 
		\begin{pmatrix}
			c_\b^{r,i}(0,0) & 0 & \cdots & 0 & 0 \\
			c_\b^{r,i}(1,0) & c_\b^{r,i}(1,1) & \cdots & 0 & 0 \\
			\vdots & \vdots & \ddots & \vdots & \vdots \\
			\vdots & \vdots & & \ddots & 0 \\
			c_\b^{r,i}(n,0) & 0 & \cdots & 0 & c_\b^{r,i}(n,n)
		\end{pmatrix}.
		\]
		\item[(2)] If $\b^n = 1$, let $\b = \gamma^j$ for some $j \in \Z$. Then $p = n - \phi(-i-j)$. In this case, the entries in the first column of $\mathcal{A}_\b^{r,i}(p+1)$ are all non-zero except for the last one, and the entries in the last row are zero except for the last one. Explicitly,
		\[
		\mathcal{A}_\b^{r,i}(p+1) = 
		\begin{pmatrix}
			c_\b^{r,i}(0,0) & 0 & \cdots & 0 & 0 \\
			c_\b^{r,i}(1,0) & c_\b^{r,i}(1,1) & \cdots & 0 & 0 \\
			\vdots & \vdots & \ddots & \vdots & \vdots \\
			\vdots & \vdots & & \ddots & 0 \\
			0 & 0 & \cdots & 0 & c_\b^{r,i}(p+1,p+1)
		\end{pmatrix}.
		\]
	\end{itemize}
\end{proposition}

We conclude this section with the following two corollaries, the first of which asserts that the standard element of a finite-dimensional simple Yetter-Drinfeld module over $H$ is unique up to a scalar factor.

\begin{corollary}\label{cor:uniquesde}
Let $V$ be a simple Yetter-Drinfeld module over $H$ with $\dim_{\k} V = p+1$ for some $p \geq 0$. Let $v \in V$ be a standard element of type $(\a,\b,x^rg^i)$, where $\a,\b \in \k^*$ and $r,i \in \Z$. If $w$ is any standard element of $V$, then $w = \lambda v$ for some $\lambda \in \k^*$.
\end{corollary}

\begin{proof}
By Lemma \ref{lem:standbasis}, the set $\{v, y \cdot v, \dots, y^p \cdot v\}$ is a basis of $V$. By Proposition~\ref{pro:spaniscomodule}(1), for each $0 \leq k \leq p$, we have
\[
g \cdot (y^k \cdot v) = \b \gamma^{-k} (y^k \cdot v).
\]
Thus, the vectors $v, y \cdot v, \dots, y^{p} \cdot v$ are eigenvectors of $g$ corresponding to distinct eigenvalues $\b\gamma^{-k}$. Since $w$ is a standard element, it is also an eigenvector of $g$. Consequently, $w = \lambda y^k \cdot v$ for some $\lambda \in \k^*$ and $0 \leq k \leq p$. Moreover, there exists an element $h \in G(H)$ such that
\[
h \otimes w = \delta(w) = \lambda \delta(y^k \cdot v) = \lambda \sum_{l=0}^{k} c_\b^{r,i}(k,l) \otimes y^l \cdot v.
\]
Comparing the coefficients of the basis elements on both sides forces $c_\b^{r,i}(k,l) = 0$ for all $0 \leq l < k$. However, by Proposition \ref{pro:summary}, the first column of the matrix $\mathcal{A}_\b^{r,i}(p)$ contains non-zero entries (specifically, $c_\b^{r,i}(k,0) \neq 0$), which contradicts the previous deduction unless $k=0$. Therefore, we must have $k=0$, which implies $w = \lambda v$.
\end{proof}

The second result states that two finite-dimensional simple Yetter-Drinfeld modules over $H$ are isomorphic if and only if they possess standard elements of the same type. We say that two Yetter-Drinfeld modules $V$ and $W$ have \textit{standard elements of the same type} if there exist $\a, \b \in \k^*$ and $r, i \in \Z$ such that both $V$ and $W$ contain a standard element of type $(\a, \b, x^r g^i)$.

\begin{corollary}\label{cor:sdedetermineV}
Let $V$ and $W$ be two finite-dimensional simple Yetter-Drinfeld modules over $H$. Then $V$ is isomorphic to $W$ if and only if they have standard elements of the same type.
\end{corollary}

\begin{proof}
If $V$ is isomorphic to $W$, it is easy to see that they have standard elements of the same type. Conversely, assume there exist $\a, \b \in \k^*$ and $r, i \in \Z$ such that $v \in V$ is a standard element of type $(\a, \b, x^r g^i)$ and $w \in W$ is a standard element of type $(\a, \b, x^r g^i)$. By Proposition \ref{pro:summary}, $\dim_{\k} V = \dim_{\k} W = p+1$ for some $p \geq 0$. By Lemma \ref{lem:standbasis}, the set $\{v, y \cdot v, \dots, y^p \cdot v\}$ is a basis of $V$, and the set $\{w, y \cdot w, \dots, y^p \cdot w\}$ is a basis of $W$.
	
Define a linear map $f: V \to W$ by $f(y^k \cdot v) = y^k \cdot w$ for all $0 \leq k \leq p$. It is straightforward to check that $f$ is an isomorphism of Yetter-Drinfeld modules. This completes the proof.
\end{proof}

\section{Classification of Simple Yetter-Drinfeld Modules over $ H $ }\label{section4}

In this section, we provide a complete classification of finite-dimensional simple Yetter-Drinfeld modules over $H$, explicitly describing their structures and determining their isomorphism classes.

The section is organized as follows. In Section \ref{subsection4.1}, we present the classification of these modules over $H$. Subsequently, in Section \ref{subsection4.2}, we demonstrate that a subset of these modules can be realized as simple Yetter-Drinfeld modules over the algebra $A(d,m,\gamma,1)$ for suitable parameters $d$ and $m$.

\subsection{Simple Yetter-Drinfeld modules over $H$}\label{subsection4.1}

By Corollary \ref{cor:sdedetermineV}, the structure of $V$ is completely determined by the type of its standard element. Consequently, to classify the finite-dimensional simple Yetter-Drinfeld modules over $H$, it suffices to construct the Yetter-Drinfeld module associated with each type of standard element.

For any scalars $\a, \b \in \k^*$ and integers $r, i \in \Z$ satisfying $\a^w = \b^n$, we define the module $V(\a, \b, x^r g^i)$ as follows:

\begin{itemize}
	\item Let $\{v_0, v_1, \dots, v_m\}$ be a $\k$-basis of $V(\a, \b, x^r g^i)$, where
	\[
	m = \begin{cases}
		n - \phi(-i-j), & \text{if } \b = \gamma^j \text{ for some } j \in \Z, \\
		n - 1,          & \text{if } \b^n \neq 1;
	\end{cases}
	\]
	
	\item The action of $x$ is defined for $0 \leq k \leq m$ by
	\[
	x \cdot v_k = \a v_k;
	\]
	
	\item The action of $g$ is defined for $0 \leq k \leq m$ by
	\begin{equation}\label{basisofV}
		g \cdot v_k = \b\gamma^{-k} v_k;
	\end{equation}
	
	\item The action of $y$ is defined for $0 \leq k \leq m$ by
	\[
	y \cdot v_k = \begin{cases}
		v_{k+1},       & 0 \leq k < m, \\
		(1-\b^n) v_0,  & k = m;
	\end{cases}
	\]
	
	\item The coaction is defined for $0 \leq k \leq m$ by
	\[
	\delta(v_k) = \sum_{l=0}^{k} c_\b^{r,i}(k,l) \otimes v_l.
	\]
\end{itemize}

\begin{remark}\label{rmk:Vtijlambda}
\normalfont Let $V = V(\a, \b, x^r g^i)$.
\begin{itemize}
	\item[(1)] 	By Remark \ref{rmk:ckl}(1), the module structure is well-defined. That is, if $x^r g^i = x^t g^s$ in $H$, then the coactions defined for $V(\a, \b, x^r g^i)$ and $V(\a, \b, x^t g^s)$ coincide.
	\item[(2)] The element $v_0$ is a standard element of type $(\a,\b,x^rg^i)$. 
	\item[(3)] The braiding induced by $V$ is of triangular type.
\end{itemize}
\end{remark}

We first prove that the modules constructed above are indeed simple Yetter-Drinfeld modules.

\begin{lemma}\label{lem:yanzhengYDmo}
For any scalars $\a, \b \in \k^*$ and integers $r, i \in \Z$ satisfying $\a^w = \b^n$, the module $V(\a, \b, x^r g^i)$ is a simple Yetter-Drinfeld module over $H$.
\end{lemma}

\begin{proof}
Let $\{v_0, v_1, \dots, v_m\}$ be the basis of $V(\a, \b, x^r g^i)$ as defined in \eqref{basisofV}, and denote $V(\a, \b, x^r g^i)$ simply by $V$. We first show that $V$ is a Yetter-Drinfeld module.
	
By Lemma~\ref{lem:comoduleequicodition} and Corollary~\ref{cor:Aijisacomatrix}, $V$ is a left $H$-comodule. It is straightforward to check that $V$ is also a left $H$-module. It remains to verify the compatibility condition \eqref{compatibilitycondition} for all $h \in H$ and $v \in V$. Define
\[
\rho(h \cdot v) = h_{(1)} v_{(-1)} S(h_{(3)}) \otimes h_{(2)} \cdot v_{(0)}.
\]
By Lemma~\ref{lem:compatiable}, it suffices to show that for all $0 \leq k \leq m$:
\begin{eqnarray}
	\delta(x \cdot v_k) &=& \rho(x \cdot v_k), \label{xcompatib} \\
	\delta(g \cdot v_k) &=& \rho(g \cdot v_k), \label{gcompatib} \\
	\delta(y \cdot v_k) &=& \rho(y \cdot v_k). \label{ycompatib}
\end{eqnarray}
	
The verification of \eqref{xcompatib} is trivial. For \eqref{gcompatib}, we compute:
\begin{eqnarray*}
		\rho (g \cdot v_k) 
		&=&\sum_{l=0}^{k} g c_\b^{r,i}(k,l) g^{-1} \otimes g \cdot v_l \\
		&=&\sum_{l=0}^{k} g \left(\lambda_\b^{i}(k,l) y^{k-l}x^r g^{i-k} \right) g^{-1} \otimes g \cdot v_l \\
		&=&\sum_{l=0}^{k} \gamma^{-(k-l)} \lambda_\b^{i}(k,l) y^{k-l}x^r g^{i-k} \otimes \b\gamma^{-l} v_l \\
		&=&\b\gamma^{-k} \sum_{l=0}^{k} c_\b^{r,i}(k,l) \otimes v_l \\
		&=&\delta(g \cdot v_k),
\end{eqnarray*}
so the condition holds.
	
To verify \eqref{ycompatib}, let us denote $y \cdot v_m$ by $v_{m+1}$. Using the definition of $c_\b^{r,i}(k,l)$, we compute
\begin{eqnarray*}
	\rho(y \cdot v_k) 
	&=&\sum_{l=0}^k \left(c_\b^{r,i}(k,l)S(y) \otimes 1 \cdot v_l + c_\b^{r,i}(k,l)S(g) \otimes y \cdot v_l + y c_\b^{r,i}(k,l)S(g) \otimes g \cdot v_l \right) \\
	&=&\sum_{l=0}^{k+1} c_\b^{r,i}(k+1,l) \otimes v_l
\end{eqnarray*}
for all $0 \leq k \leq m$. This establishes \eqref{ycompatib} for $0 \leq k < m$.
	
For the case $k=m$, we consider the following two cases.
	
\noindent \textbf{Case 1: $\b=\gamma^j$ for some $j \in \Z$.} \\
In this case, $m=n-\phi(-i-j)$ and $v_{m+1} = (1-\b^n)v_0 = 0$. By Proposition \ref{pro:summary}(2), we have
\[
\rho(y \cdot v_m) = \sum_{l=0}^{m+1} c_\b^{r,i}(m+1,l) \otimes v_l = 0 = \delta(y \cdot v_m).
\]
	
\noindent \textbf{Case 2: $\b^n \neq 1$.} \\
In this case, $m=n-1$ and $v_n = (1-\b^n)v_0$. By Proposition \ref{pro:summary}(1), we have
\begin{eqnarray*}
	\rho(y \cdot v_{n-1})
	&=&\sum_{l=0}^{n} c_\b^{r,i}(n,l) \otimes v_l \\
	&=&c_\b^{r,i}(n,0) \otimes v_0 + c_\b^{r,i}(n,n) \otimes v_n \\
	&=&\lambda_\b^{i}(n,0) y^n x^r g^{i-n} \otimes v_0 + x^r g^{i-n} \otimes (1-\b^n)v_0 \\
	&=&(1-\b^n) x^r g^i \otimes v_0 \\
	&=&\delta(y \cdot v_{n-1}).
\end{eqnarray*}
	
Therefore, the compatibility condition holds, and $V$ is a Yetter-Drinfeld module.
	
Finally, we prove that $V$ is simple. Let $W$ be a non-zero Yetter-Drinfeld submodule of $V$; we show that $W = V$. By Lemma \ref{lem:existenceofse}, we can choose a standard element $w$ of $W$. Similar to the argument in Corollary \ref{cor:uniquesde}, we can prove that $w = \lambda v_0$ for some $\lambda \in \k^*$. Since $v_0$ generates $V$ as a module (by the construction of the basis), it follows that $W = V$.
\end{proof}

Proposition \ref{prop:simplemustfinite} ensures that every simple Yetter-Drinfeld module $V$ over $H$ is finite-dimensional. Furthermore, Lemma \ref{lem:existenceofse} and Corollary \ref{cor:uniquesde} guarantee that each such module contains a standard element, which is unique up to a scalar factor. Since the structure of $V$ is completely determined by the type of its standard element (Corollary \ref{cor:sdedetermineV}), we now have a complete classification of simple Yetter-Drinfeld modules over $H$ alongside their isomorphism criteria. We summarize these results in the following theorem.

\begin{theorem}\label{thm:YDmodule}
Let $H = B(n,w,\gamma)$.
\begin{itemize}
	\item[(1)] Every simple Yetter-Drinfeld module over $H$ is isomorphic to $V(\a,\b,x^rg^i)$ for some scalars $\a, \b \in \k^*$ and integers $r, i \in \Z$ satisfying $\a^w = \b^n$.
	\item[(2)] Two such modules $V(\a,\b,x^rg^i)$ and $V(\a',\b',x^{r'}g^{i'})$ are isomorphic if and only if
	\[
	\a = \a', \quad \b = \b', \quad \text{and} \quad x^r g^i = x^{r'} g^{i'}.
	\]
\end{itemize}
\end{theorem}

\subsection{Connection to simple Yetter-Drinfeld modules over $A(d,m,\gamma,1)$}\label{subsection4.2}

In this subsection, we prove that when $\a$ and $\b$ are roots of unity, the module $V(\a,\b,x^rg^i)$ can be realized as a simple Yetter-Drinfeld module over $A(d,m,\gamma,1)$ for suitable choices of $d$ and $m$.

Recall that $R_\infty$ denotes the group of all roots of unity. Let $\mathcal{T}$ denote the set of types of all standard elements. We partition $\mathcal{T}$ into two subsets $\mathcal{T}_1$ and $\mathcal{T}_2$ as follows:
\begin{eqnarray}
	\mathcal{T} &=& \left\{ (\a,\b,x^r g^i) \mid \a,\b \in \k^*, \, r, i \in \Z, \text{ and } \a^w = \b^n \right\}, \label{indexT} \\
	\mathcal{T}_1 &=& \left\{ (\a,\b,x^r g^i) \in \mathcal{T} \mid \a,\b \in R_{\infty} \right\}, \label{indexT1} \\
	\mathcal{T}_2 &=& \left\{ (\a,\b,x^r g^i) \in \mathcal{T} \mid \a,\b \notin R_{\infty} \right\}. \label{indexT2}
\end{eqnarray}
Note that the condition $\a^w = \b^n$ implies that $\a \in R_\infty$ if and only if $\b \in R_\infty$. Consequently, there are no mixed cases, and we have the decomposition $\mathcal{T} = \mathcal{T}_1 \cup \mathcal{T}_2$.

Note that if $\b$ is a $d$-th root of unity, it is also a $dn$-th root of unity. In other words, we can choose $d$ such that $n \mid d$. 

\begin{lemma}\label{lem:relization}
For any $(\a,\b,x^r g^i) \in \mathcal{T}_1$, assume that $\a$ is an $m$-th root of unity and $\b$ is a $d$-th root of unity such that $n \mid d$. Then $V(\a,\b,x^r g^i)$ is a simple Yetter-Drinfeld module over $A(d,m,\gamma,1)$.
\end{lemma}

\begin{proof}
The proof is identical to that of Lemma \ref{lem:yanzhengYDmo}.
\end{proof}

We conclude this section with the following remark.

\begin{remark}\hfil\label{rmk:nonunique_realization}
\normalfont
	\begin{itemize}
		\item[(1)] Note that if $\a$ is an $m$-th root of unity, it is also a $2m$-th root of unity; similarly, if $\b$ is a $d$-th root of unity, it is also a $2d$-th root of unity. Thus, by Lemma \ref{lem:relization}, $V(\a,\b,x^r g^i)$ is also a simple Yetter-Drinfeld module over $A(2d,2m,\gamma,1)$. This demonstrates that the module admits multiple realizations (i.e., the realization is not unique).
		
		\item[(2)] If $(\a,\b,x^r g^i) \in \mathcal{T}_2$, then $V(\a,\b,x^r g^i)$ cannot be realized as a Yetter-Drinfeld module over any finite-dimensional pointed Hopf algebra. The reason is straightforward: group-like elements of a finite-dimensional pointed Hopf algebra have finite order. Consequently, their action on any eigenvector (such as a standard element) must be multiplication by a root of unity. This contradicts the definition of $\mathcal{T}_2$, where $\a$ and $\b$ are not roots of unity.
	\end{itemize}
\end{remark}

\section{Finite-dimensional Nichols algebras over the simple modules $ V(\a,\b,x^rg^i) $ }\label{section5}

In this section, we classify all Yetter-Drinfeld modules $V(\a, \b, x^r g^i)$ such that the associated Nichols algebra $\B(V(\a, \b, x^r g^i))$ is finite-dimensional. We divide our discussion into two main cases, depending on whether the type $(\a, \b, x^r g^i)$ belongs to $\mathcal{T}_1$ or $\mathcal{T}_2$.

\subsection{The case $(\a, \b, x^r g^i) \in \mathcal{T}_1$}

We begin by addressing the first case.  Assume that $\a$ is a primitive $m$-th root of unity and $\b$ is a $d$-th root of unity such that $n \mid d$. Throughout this subsection, we fix the integers $d$ and $m$. Let $\xi$ be a primitive $d$-th root of unity. Let $G = \langle g \rangle \times \langle x \rangle$ be the direct product of cyclic groups of orders $d$ and $m$, generated by $g$ and $x$, respectively. Let $A_0$ and $A_1$ denote the algebras $A(d,m,\gamma,0)$ and $A(d,m,\gamma,1)$, respectively, and let $V = V(\a, \b, x^r g^i)$.

We first consider the case $n=1$. In this setting, $V = \k\{v\}$ is a one-dimensional Yetter-Drinfeld module with braiding given by $c(v \otimes v) = \a^r \b^i v \otimes v$. Since $\a$ and $\b$ are roots of unity, the scalar $\a^r \b^i$ is also a root of unity. Consequently, the Nichols algebra $\B(V)$ is finite-dimensional if and only if $\a^r \b^i \neq 1$ (see Remark \ref{rmk:exampledim1}).

Now assume $n \geq 2$. Recall from Lemma \ref{lem:bosination} and Lemma \ref{lem:2cocycle} that $A_0$ can be realized as a bosonization of a Nichols algebra over $\k G$, whereas $A_1$ is a 2-cocycle deformation of $A_0$. We adopt the notation used therein.

By Lemma \ref{lem:relization}, $V$ is a simple Yetter-Drinfeld module over $A_1$. Consequently, determining whether the Nichols algebra $\B(V)$ is finite-dimensional in the category ${}^H_H \mathcal{YD}$ reduces to the same question in the category ${}^{A_1}_{A_1} \mathcal{YD}$. By invoking the results of Andruskiewitsch and Angiono \cite{AA20}, we can transform this into the problem of determining whether a Nichols algebra of diagonal type is finite-dimensional. This classification was completed by Heckenberger using root systems \cite{H08,H09}.

We begin by recalling the bijective correspondence between simple Yetter-Drinfeld modules over $\k G$ and those over $A_0 = \B(V_0) \biproduct \k G$, as established in \cite{AA20, AHS10}.

The category $\YD{\k G}$ is semisimple, and its simple objects are parameterized by tuples of integers $(d_1, d_2, m_1, m_2)$ satisfying $0 \leq d_1, d_2 < d$ and $0 \leq m_1, m_2 < m$. Explicitly, the simple object $\lambda_{d_1, m_1}^{d_2, m_2}$ associated with the tuple $(d_1, d_2, m_1, m_2)$ is the one-dimensional vector space $\k\{v\}$ generated by an element $v$, with the action and coaction given respectively by:
\begin{itemize}
	\item The action is defined by
	\begin{equation*}
		g \cdot v = \xi^{d_1} v \quad \text{and} \quad x \cdot v = \a^{m_1} v;
	\end{equation*}
	
	\item The coaction is defined by
	\begin{equation*}
		\delta(v) = g^{d_2} x^{m_2} \otimes v.
	\end{equation*}    
\end{itemize}

Observe that the module $V_0$ described in Lemma \ref{lem:bosination} corresponds to $\lambda_{d_1, 0}^{1, 0}$, provided that $\xi^{d_1} = \gamma$. Furthermore, the notation $\lambda_{d_1, m_1}^{d_2, m_2}$ is well-defined for any integers $d_1, d_2, m_1, m_2 \in \Z$. Moreover, we have the identification
\[
\lambda_{d_1, m_1}^{d_2, m_2} = \lambda_{d_1', m_1'}^{d_2', m_2'}
\]
whenever $d_1 \equiv d_1' \pmod d$, $d_2 \equiv d_2' \pmod d$, $m_1 \equiv m_1' \pmod m$, and $m_2 \equiv m_2' \pmod m$.

Given any simple Yetter-Drinfeld module $\lambda$ over $\k G$, take $W = V_0 \oplus \lambda$. The Nichols algebra $\B(W)$ admits a decomposition $\B(W) = \mathcal{K} \biproduct \B(V_0)$, where $\mathcal{K}$ is a Hopf algebra in the category $\YD{A_0}$. Define
\[
L(\lambda) = \operatorname{ad}_c \B(V)(\lambda).
\]
Then $L(\lambda) \subset \mathcal{K}$ is a simple Yetter-Drinfeld module over $A_0$; see \cite[Section 3.3]{AHS10}. Moreover, we have the following results.

\begin{proposition}\label{pro:todiagonal}
Let $n\geq 2$. Then
\begin{itemize}
	\item[(1)] The assignment $\lambda_{d_1, m_1}^{d_2, m_2} \mapsto L(\lambda_{d_1, m_1}^{d_2, m_2})$ establishes a bijective correspondence between the simple objects of $\YD{\k G}$ and the simple objects of $\YD{A_0}$, for all $0 \leq d_1, d_2 < d$ and $0 \leq m_1, m_2 < m$.
	\item[(2)] The Nichols algebra $\B(L(\lambda_{d_1, m_1}^{d_2, m_2}))$   is finite-dimensional if and only if  $\B(V_0 \oplus  \lambda_{d_1, m_1}^{d_2, m_2} )$ is finite-dimensional.
\end{itemize}
\end{proposition}

\begin{proof}
Part (1) follows from \cite[Lemma 3.3 and Proposition 3.5]{AHS10} and \cite[Proposition 2.9]{AA20}. Part (2) follows from \cite[Proposition 2.10]{AA20}.
\end{proof}

Using the 2-cocycle described in Lemma \ref{lem:2cocycle}, we can reduce the problem of determining the finite dimensionality of $\B(V)$ to determining that of a Nichols algebra of diagonal type. We adopt the notation introduced above.

\begin{corollary}\label{cor:todiagonal}
Let $n \geq 2$. If $(\a, \b, x^r g^i) \in \mathcal{T}_1$, then $\B(V(\a, \b, x^r g^i))$ is finite-dimensional if and only if $\B(W(\a,\b,x^rg^i))$ is finite-dimensional, where $W(\a,\b,x^rg^i) = V_0 \oplus V_1$. Here, $V_0$ is as in Lemma \ref{lem:bosination}, while $V_1 = \k\{v\}$ denotes the Yetter-Drinfeld module over $\k G$ defined by
\[
g \cdot v = \b^{-1} v, \quad x \cdot v = \a v, \quad \text{and} \quad \delta(v) = x^r g^{-i} \otimes v.
\]
\end{corollary}

\begin{proof}
By Theorem \ref{thm:cocycle-equivalence} and Proposition \ref{pro:todiagonal}, it suffices to establish an isomorphism $(L(V_1))_{\sigma} \simeq V(\a,\b,x^r g^i)$.
	
Recall that $V_1 = \k\{v\}$. We examine the action and coaction on the element $v$ when regarded as an element of $L(V_1) \subset \mathcal{K}$. As an object in $\YD{\B(V_0) \biproduct \k G}$, a straightforward calculation shows that:
\[
(1 \# g )\cdot v = \b^{-1} v, \quad (1 \# x) \cdot v = \a v, \quad \text{and} \quad \delta(v) = (1 \# x^r g^{-i}) \otimes v.
\]
Identifying $A_0$ with $\B(V_0) \biproduct \k G$ via the isomorphism $\psi$ from Lemma \ref{lem:bosination}, the action and coaction on $v$ as an element of a Yetter-Drinfeld module over $A_0$ become:
\[
g\cdot v = \b v, \quad x \cdot v = \a v, \quad \text{and} \quad \delta(v) = x^r g^i \otimes v.
\]
Since $L(V_1)$ is simple, the twisted module $(L(V_1))_{\sigma}$ remains simple over $A_1$. We verify the twisted structure on the element $v$:
\[
g\cdot_{\sigma} v = \b v, \quad x \cdot_{\sigma} v = \a v, \quad \text{and} \quad \delta(v) = x^r g^i \otimes v.
\]
This shows that $(L(V_1))_{\sigma}$ contains a standard element of type $(\a, \b, x^r g^i)$. By an argument analogous to Corollary \ref{cor:sdedetermineV}, the isomorphism class of such a simple module is uniquely determined by this standard element. Therefore, $(L(V_1))_{\sigma} \simeq V(\a,\b,x^r g^i)$.
\end{proof}

Observe that $W(\a,\b,x^rg^i)$ is a braided vector space of diagonal type with respect to the basis $\{y, v\}$. Consequently, the finite dimensionality of its Nichols algebra can be determined using Heckenberger's classification \cite{H08,H09}, which establishes a correspondence between finite-dimensional Nichols algebras of diagonal type and certain generalized Dynkin diagrams. We now proceed to apply this theory.

The braiding $c$ of $W(\a,\b,x^rg^i)$ on the basis elements is given by:
\begin{equation}
	\begin{array}{lclcl}
		c(y\otimes y) &=& \gamma y\otimes y, & \quad & c(y\otimes v) = \b^{-1} v \otimes y, \\
		c(v\otimes y) &=& \gamma^{-i} y \otimes v, & \quad & c(v\otimes v) = \a^r \b^i v\otimes v.
	\end{array}
\end{equation}
The corresponding braiding matrix $\mathbf{q} = (q_{ij})_{2\times 2}$ is
\begin{equation}
	\mathbf{q} = \left(\begin{array}{cc}
		\gamma & \b^{-1} \\
		\gamma^{-i} & \a^r \b^i
	\end{array}\right).
\end{equation}
The generalized Dynkin diagram $D(\a,\b,x^rg^i)$ is therefore given by:
\begin{equation}\label{Dynkindiagram}
	\begin{array}{c}
		\begin{tikzpicture}[baseline=(current bounding box.center), scale=1.2]
			\node[circle, draw, fill=white, inner sep=2pt, label=above:$\gamma$] (A) at (0,0) {};
			\node[circle, draw, fill=white, inner sep=2pt, label=above:$\a^r \b^i$] (B) at (2,0) {};
		\end{tikzpicture}
		\qquad \text{or} \qquad
		\begin{tikzpicture}[baseline=(current bounding box.center), scale=1.2]
			\node[circle, draw, fill=white, inner sep=2pt, label=above:$\gamma$] (A) at (0,0) {};
			\node[circle, draw, fill=white, inner sep=2pt, label=above:$\a^r \b^i$] (B) at (2,0) {};
			\draw (A) -- node[above] {$\b^{-1} \gamma^{-i}$} (B);
		\end{tikzpicture}
	\end{array}
\end{equation}
depending on whether the product $q_{12}q_{21} = \gamma^{-i}\b^{-1}$ is equal to $1$ or not.

\begin{remark}
\normalfont
Note that if $\a^r = 1$ and $\b = \gamma^j$, then the generalized Dynkin diagram $D(\a, \b, x^r g^i)$ coincides with the diagram $D_{i,j}$ given in \cite[Section 5.2]{ZLY25}.
\end{remark}

The generalized Dynkin diagrams corresponding to arithmetic root systems of
rank 2 are listed in \cite[Table 1]{H08}. These diagrams depend on fixed parameters $ q,\zeta \in \k^* $.  Following the notation in \cite{MBG21}, we refer to these as Heckenberger diagrams. Each diagram is indexed as $H_{k,l}$, where $k$ denotes the row number and $l$ denotes the position within that row from left to right.
For example, $ H_{4,1} $ denotes the diagram 
\begin{tikzpicture}[
	scale=0.7, baseline=-3pt,
	every node/.style={     
		circle,
		draw=black,
		fill=white,
		inner sep=1pt,
		minimum size=4pt   
	},
	label distance=0.5mm
	]
	\node[label=above:{\scriptsize $q$}] (A) at (0,0) {};
	\node[label=above:{\scriptsize $q^{2}$}] (B) at (1.5,0) {};

	\draw (A) -- node[draw=none, fill=none, rectangle, above, scale=0.8] {$q^{-2}$} (B);
\end{tikzpicture}
, with $ q\in \k^* \textbackslash \{-1,1\} $. The set of indices of Heckenberger
diagrams is the following:
\begin{eqnarray*}
	\mathcal{I} \ := & \{& (1,1),(2,1),(3,1),(3,2),(4,1),(5,1),(5,2),(6,1),(6,2),(7,1),(7,2), \\   
	&&(8,1), (8,2),(8,3),(8,4),(8,5),(9,1),(9,2),(9,3),(10,1), (10,2),(10,3),     \\
	&& (11,1),(12,1),(12,2),(12,3),(13,1),(13,2),(13,3),(13,4),(14,1),(14,2),   \\
	&& (15,1),(15,2),(15,3),(15,4),(16,1),(16,2),(16,3),(16,4), (17,1),(17,2) \  \} .
\end{eqnarray*}

By invoking Corollary \ref{cor:todiagonal} and \cite[Corollary 6]{H08}, we can determine whether $\B(V(\a,\b,x^rg^i))$ is finite-dimensional via $D(\a,\b,x^rg^i)$, as formalized in the following lemma.

\begin{lemma}\label{lem:WandD}
Let $n \geq 2$ and $(\a, \b, x^r g^i) \in \mathcal{T}_1$.
The Nichols algebra $\B(V(\a,\b,x^rg^i))$ is finite-dimensional if and only if
$D(\a,\b,x^rg^i) = H_{k,l}$ for some $(k,l) \in \mathcal{I}$ and $\a^r \b^i \neq 1$.
\end{lemma}

\begin{remark}\hfill
\normalfont
	\begin{itemize}
		\item [(1)] In fact, the condition $\a^r \b^i \neq 1$ is automatically satisfied whenever $(k,l) \neq (1,1)$. In such cases, $\B(V(\a,\b,x^rg^i))$ is finite-dimensional if and only if $D(\a,\b,x^rg^i) = H_{k,l}$.
		\item [(2)] Recall that we initially fixed the Hopf algebra $H = B(n, w, \gamma)$ at the end of Section \ref{section2.5}, and our preceding discussions have been based on this convention. However, to classify all finite-dimensional Nichols algebras over such Hopf algebras, we must also consider the structural parameters $n$, $w$, and $\gamma$. To this end, the notation $V(n, w, \gamma, \a, \b, x^r g^i)$ is introduced here to explicitly denote the Yetter-Drinfeld module $V(\a, \b, x^r g^i)$ over the Hopf algebra $B(n, w, \gamma)$, as constructed in Section \ref{subsection4.1}. Correspondingly, $D(n, w, \gamma, \a, \b, x^r g^i)$ represents the associated generalized Dynkin diagram, as introduced below.
	\end{itemize}
\end{remark}

Recall that for any integer $m \geq 1$, $R_m$ denotes the set of primitive $m$-th roots of unity, and $R_\infty$ denotes the group of all roots of unity. Our next task is to determine all 6-tuples $(n, w, \gamma, \a, \b, x^r g^i)$ such that the associated diagram $D(n, w, \gamma, \a, \b, x^r g^i)$ coincides with a Heckenberger diagram. Specifically, we consider the 6-tuples satisfying the following fundamental conditions: $n \geq 2$ and $w \geq 1$ are integers, $\gamma \in R_n$, $r, i \in \Z$, and $\a, \b \in R_\infty$ with the relation $\a^w = \b^n$.

Our classification strategy proceeds as follows: First, we determine $\gamma$ using the label of a vertex, which consequently determines the integer $n$. Second, for any chosen integer $i \in \Z$, the scalar $\b$ is uniquely determined by the label of the edge. Finally, we characterize the remaining parameters $\a$, $r$, and $w$. Note that we do not assign specific values to $\a$, $r$, and $w$; instead, we establish the algebraic equations that they must satisfy.

We now apply this strategy to the specific Heckenberger diagrams $H_{4,1}, H_{6,2}, H_{10,1}, H_{11,1},$ and $H_{16,1}$. We examine each case individually.

\subsubsection*{The diagram $H_{4,1}$}
The diagram $H_{4,1}$ is given by:
\begin{equation*}
	\begin{tikzpicture}[baseline=(current bounding box.center), scale=1.2]
		\node[circle, draw, fill=white, inner sep=2pt, label=above:$q$] (A) at (0,0) {};
		\node[circle, draw, fill=white, inner sep=2pt, label=above:$q^2$] (B) at (2,0) {};
		\draw (A) -- node[above] {$q^{-2}$} (B);
	\end{tikzpicture}
	\qquad \text{with } q \in \k^* \setminus \{-1, 1\}.
\end{equation*}
The vertices are labeled by distinct scalars $q$ and $q^2$. Consequently, there are two possible identifications for $\gamma$.

\textbf{Case 1:} $\gamma = q$.
Since $q \in \k^* \setminus \{-1, 1\}$, we have $\gamma \neq \pm 1$, which implies $n \geq 3$.
For any chosen integer $i \in \Z$, the condition on the edge label $\b^{-1} \gamma^{-i} = q^{-2}$ becomes $\b^{-1} \gamma^{-i} = \gamma^{-2}$, which yields
\[
\b = \gamma^{2-i}.
\]
With $n, i$, and $\b$ determined, we consider the remaining parameters $\a, w, r$.
Substituting $\b = \gamma^{2-i}$ into the relations $\a^w = \b^n$ and $\a^r \b^i = q^2 = \gamma^2$, we obtain the following conditions:
\[
\a^w = \b^n = 1 \quad \text{and} \quad \a^r = \gamma^{i^2 - 2i + 2}.
\]

\textbf{Case 2:} $\gamma = q^2$.
In this setting, $n \geq 2$. Similarly, for any chosen $i \in \Z$, the edge label condition $\b^{-1} \gamma^{-i} = q^{-2}$ becomes $\b^{-1} \gamma^{-i} = \gamma^{-1}$, which implies
\[
\b = \gamma^{1-i}.
\]
Finally, we determine $\a, w, r$. The relation $\a^r \b^i = q$ implies $(\a^r \b^i)^2 = q^2 = \gamma$. Substituting $\b = \gamma^{1-i}$ into the relations $\a^w = \b^n$ and $\a^{2r} \b^{2i} = \gamma$, we arrive at:
\[
\a^w = \b^n = 1 \quad \text{and} \quad \a^{2r} = \gamma^{2i^2 - 2i + 1}.
\]

\subsubsection*{The diagram $H_{6,2}$}
The diagram $H_{6,2}$ is given by:
\begin{equation*}
	\begin{tikzpicture}[baseline=(current bounding box.center), scale=1.2]
		\node[circle, draw, fill=white, inner sep=2pt, label=above:$\zeta$] (A) at (0,0) {};
		\node[circle, draw, fill=white, inner sep=2pt, label=above:$\zeta q^{-1}$] (B) at (2,0) {};
		\draw (A) -- node[above] {$\zeta^{-1}q$} (B);
	\end{tikzpicture}
	\qquad \text{with } \zeta \in R_3, \ q \in \k^* \setminus \{1, \zeta, \zeta^2\}.
\end{equation*}
Obviously, $\zeta \neq \zeta q^{-1}$ since $q \neq 1$. Thus, the vertices are distinct, and there are two possible identifications for $\gamma$.

\textbf{Case 1:} $\gamma = \zeta$.
Since $\gamma \in R_3$, we immediately have $n=3$.
For any chosen integer $i \in \Z$, the condition on the edge label is $\b^{-1} \gamma^{-i} = \gamma^{-1}q$.
Since $\b, \gamma$ are roots of unity, $q$ must also be a root of unity. Given the constraint $q \in \k^* \setminus \{1, \zeta, \zeta^2\}$, we infer that $q \in R_m$ where $m=2$ or $m \geq 4$. This is equivalent to requiring $\b \in R_m$ where $m=2$ or $m \geq 4$.
Finally, we determine the relations for $\a$. Substituting the expression for $\b$ into the vertex equation, we obtain the following classification for the parameters:
\[
\a^w = \b^3 \quad \text{and} \quad \a^r = \b^{1-i} \gamma^i.
\]

\textbf{Case 2:} $\gamma = \zeta q^{-1}$.
In this setting, we have $n=2$ or $n \geq 4$.
Similarly, for any chosen $i \in \Z$, the edge label condition $\b^{-1} \gamma^{-i} = \zeta^{-1}q = \gamma^{-1}$ implies
\[
\b = \gamma^{1-i}.
\]
Finally, using the label of the first vertex $\a^r \b^i = \zeta$, and substituting $\b = \gamma^{1-i}$, we deduce that $\a^r \gamma^{i-i^2} = \zeta$. Thus, the parameters satisfy:
\[
\a^w = 1 \quad \text{and} \quad \a^r \gamma^{i-i^2} \in R_3.
\]

\subsubsection*{The diagram $H_{10,1}$}
The diagram $H_{10,1}$ is given by:
\begin{equation*}
	\begin{tikzpicture}[baseline=(current bounding box.center), scale=1.2]
		\node[circle, draw, fill=white, inner sep=2pt, label=above:$-\zeta$] (A) at (0,0) {};
		\node[circle, draw, fill=white, inner sep=2pt, label=above:$\zeta^3$] (B) at (2,0) {};
		\draw (A) -- node[above] {$\zeta^{-2}$} (B);
	\end{tikzpicture}
	\qquad \text{with } \zeta \in R_9.
\end{equation*}
Obviously, the vertices are distinct because $\zeta \in R_9$. Thus, there are two cases.

\textbf{Case 1:} $\gamma = -\zeta$.
Obviously, $n=18$.
For any chosen integer $i \in \Z$, the edge label condition is $\b^{-1} \gamma^{-i} = \zeta^{-2}$.
Since $\gamma = -\zeta$, we have $\zeta^{-2} = (-\gamma)^{-2} = \gamma^{-2}$. Thus, the condition simplifies to $\b^{-1} \gamma^{-i} = \gamma^{-2}$, which yields
\[
\b = \gamma^{2-i}.
\]
We now determine $\a$. Using the label of the second vertex $\a^r \b^i = \zeta^3$, and observing that $\zeta^3 = (-\gamma)^3 = -\gamma^3 = \gamma^{12}$ (since $\gamma^9 = -1$), we substitute $\b = \gamma^{2-i}$ to get:
\[
\a^w = 1 \quad \text{and} \quad \a^r = \gamma^{i^2-2i+12}.
\]

\textbf{Case 2:} $\gamma = \zeta^3$.
Since $\zeta \in R_9$, we have $\gamma \in R_3$, so $n=3$.
Similarly, for any chosen $i \in \Z$, the edge label condition $\b^{-1} \gamma^{-i} = \zeta^{-2}$ implies
\[
\b = \zeta^{2-3i}.
\]
Finally, using the label of the first vertex $\a^r \b^i = -\zeta$, we substitute $\b$ to find $\a$.
Note that $\a^w = \b^3 = (\zeta^{2-3i})^3 = \zeta^{6-9i} = \zeta^6$.
Thus, the parameters are classified as:
\[
\a^w = \zeta^6 \quad \text{and} \quad \a^r = -\zeta^{3i^2-2i+1}.
\]

\subsubsection*{The diagram $H_{11,1}$}
The diagram $H_{11,1}$ is given by:
\begin{equation*}
	\begin{tikzpicture}[baseline=(current bounding box.center), scale=1.2]
		\node[circle, draw, fill=white, inner sep=2pt, label=above:$q$] (A) at (0,0) {};
		\node[circle, draw, fill=white, inner sep=2pt, label=above:$q^3$] (B) at (2,0) {};
		\draw (A) -- node[above] {$q^{-3}$} (B);
	\end{tikzpicture}
	\qquad \text{with } q \in \k^* \setminus \{-1, 1\}, \ q \notin R_3.
\end{equation*}
Obviously, the vertices are distinct. Thus, there are two cases.

\textbf{Case 1:} $\gamma = q$.
Since $q \notin R_3$ and $q \neq \pm 1$, we have $n \geq 4$.
For any chosen integer $i \in \Z$, the edge label condition implies
\[
\b = \gamma^{3-i}.
\]
Using the label of the second vertex, we obtain the classification:
\[
\a^w = 1 \quad \text{and} \quad \a^r = \gamma^{i^2-3i+3}.
\]

\textbf{Case 2:} $\gamma = q^3$.
In this setting, $n \geq 2$. Similarly, for any chosen $i \in \Z$, the edge label condition implies
\[
\b = \gamma^{1-i}.
\]
Using the label of the first vertex $\a^r \b^i = q$, which implies $(\a^r \b^i)^3 = \gamma$, we arrive at:
\[
\a^w = 1 \quad \text{and} \quad \a^{3r} = \gamma^{3i^2 - 3i + 1}.
\]

\subsubsection*{The diagram $H_{16,1}$}
The diagram $H_{16,1}$ is given by:
\begin{equation*}
	\begin{tikzpicture}[baseline=(current bounding box.center), scale=1.2]
		\node[circle, draw, fill=white, inner sep=2pt, label=above:$-\zeta$] (A) at (0,0) {};
		\node[circle, draw, fill=white, inner sep=2pt, label=above:$\zeta^5$] (B) at (2,0) {};
		\draw (A) -- node[above] {$-\zeta^{-3}$} (B);
	\end{tikzpicture}
	\qquad \text{with } \zeta \in R_{15}.
\end{equation*}
Obviously, the vertices are distinct because $\zeta \in R_{15}$. Thus, there are two cases.

\textbf{Case 1:} $\gamma = -\zeta$.
Obviously, $n=30$.
For any chosen integer $i \in \Z$, the edge label condition is $\b^{-1} \gamma^{-i} = -\zeta^{-3}$.
Since $\gamma = -\zeta$, we have $-\zeta^{-3} = -(-\gamma)^{-3} = -(-1)^{-3}\gamma^{-3} = \gamma^{-3}$. Thus, the condition simplifies to $\b^{-1} \gamma^{-i} = \gamma^{-3}$, which yields
\[
\b = \gamma^{3-i}.
\]
We now determine $\a$. Using the label of the second vertex $\a^r \b^i = \zeta^5$, and observing that $\zeta^5 = (-\gamma)^5 = -\gamma^5$, we substitute $\b = \gamma^{3-i}$ to get:
\[
\a^w = 1 \quad \text{and} \quad \a^r = -\gamma^{i^2-3i+5}.
\]

\textbf{Case 2:} $\gamma = \zeta^5$.
First, the structure of the diagram implies $\gamma = \zeta^5$ and the edge label condition is $\b^{-1} \gamma^{-i} = -\zeta^{-3}$, where $\zeta \in R_{15}$.
This is equivalent to the following  conditions:
\[
\gamma \in R_3 \quad \text{and} \quad \b^{-1} \gamma^{-i} \in R_{10}.
\]
One direction of this equivalence is trivial. Conversely, given $\gamma \in R_3$ and $\b^{-1} \gamma^{-i} \in R_{10}$, one can explicitly recover $\zeta$ by defining $\zeta = -\gamma^2 (\b^{-1} \gamma^{-i})^3$, which perfectly satisfies $\zeta \in R_{15}$, $\gamma = \zeta^5$, and $\b^{-1} \gamma^{-i} = -\zeta^{-3}$.

From $\gamma \in R_3$, we immediately have $n=3$. Finally, we express $\a$ purely in terms of $\gamma$ and $\b$. 
Observing from our construction that $\zeta = -\gamma^2 (\b^{-1} \gamma^{-i})^3 = -\b^{-3} \gamma^{2-3i}$, we substitute this into the first vertex label equation $\a^r \b^i = -\zeta$ to obtain:
\[
\a^w = \b^n = \b^3 \quad \text{and} \quad \a^r = \b^{-3-i} \gamma^{2-3i}.
\]

The analysis for the remaining Heckenberger diagrams proceeds similarly. We now summarize the complete classification. 

In the following lemma, we classify all 6-tuples $(n, w, \gamma, \a, \b, x^r g^i)$ such that the associated Nichols algebra $\B(V(n, w, \gamma, \a, \b, x^r g^i))$ is finite-dimensional. We organize this classification into Tables \ref{tab:dim1_classification} through \ref{tab:dim6_classification} based on the dimension of the underlying Yetter-Drinfeld module $V(n, w, \gamma, \a, \b, x^r g^i)$.

In these tables, the first column specifies the conditions on $n$ and $\gamma$ (typically, this restricts only the value of $n$, allowing $\gamma$ to be any primitive $n$-th root of unity). The second column outlines the conditions on $\b$ and $ i $, while the third column details the algebraic relations that $\a,w$ and $r$ must satisfy.
Whenever a 6-tuple $(n, w, \gamma, \a, \b, x^r g^i)$ simultaneously satisfies the conditions in these first three columns, the Nichols algebra $\B(V(n, w, \gamma, \a, \b, x^r g^i))$ is finite-dimensional. Furthermore, its corresponding diagram coincides with the Heckenberger diagram $H_{k,l}$ for any pair $(k,l)$ listed in the fourth column.

\begin{lemma}\label{lem:finiten-nichols-classification}
For all pairs $(k,l) \in \mathcal{I}$, there exists a 6-tuple $(n, w, \gamma, \a, \b, x^r g^i)$ satisfying $n, w \geq 1$, $\gamma \in R_n$, and $(\a, \b, x^r g^i) \in \mathcal{T}_1$ such that $D(n, w, \gamma, \a, \b, x^r g^i) = H_{k,l}$. Moreover, for any such 6-tuple, it holds that $\dim_{\k} \B(V(n, w, \gamma, \a, \b, x^r g^i)) < \infty$ if and only if $\a^r \b^i \neq 1$. The complete classification of 6-tuples $(n, w, \gamma, \a, \b, x^r g^i)$ yielding finite-dimensional Nichols algebras is provided in Tables \ref{tab:dim1_classification}--\ref{tab:dim6_classification}.
\end{lemma}

% =========================================================
% Table 1: 对应 YD 模维数为 1 的情况 (d=1)
% =========================================================
\begingroup
\small
\setlength{\tabcolsep}{4pt}
\setlength{\LTleft}{-20cm plus 1fill}
\setlength{\LTright}{-20cm plus 1fill}
\renewcommand{\arraystretch}{1.8}
\begin{longtable}{ >{$}c<{$} | >{$}c<{$} | >{$}c<{$} | >{$}c<{$} }
	\caption[Classification of finite-dimensional Nichols algebras for dimension 1]{Classification of finite-dimensional Nichols algebras for $\dim_{\protect\k} V = 1$}
	\label{tab:dim1_classification} \\
	\hline
	\makebox[2.6cm]{$n,\gamma$} & \makebox[2.8cm]{$\b, i$} & \makebox[5.4cm]{$\a, w, r$} & \makebox[1.6cm]{$(k,l)$} \\
	\hline\hline
	\endfirsthead
	
	\caption[]{Classification of finite-dimensional Nichols algebras for $\dim_{\protect\k} V = 1$ (continued)} \\
	\hline
	n,\gamma & \b, i & \a, w, r & (k,l) \\
	\hline\hline
	\endhead
	
	\hline
	\endfoot
	
	\hline
	\endlastfoot
	
	n =1 & 
	\b \in R_\infty & 
	\a^w = \b^n, \quad \a^r \b^i \in R_\infty \setminus \{1\} & 
	\\
	\hline
	% H_{1,1} (Unlinked Case)
	n \geq 2 & 
	\b = \gamma^{-i} & 
	\a^w = 1, \quad \a^r \gamma^{-i^2} \in R_\infty \setminus \{1\} & 
	(1, 1) \\
\end{longtable}
\endgroup

% =========================================================
% Table 2: 对应 YD 模维数为 2 的情况 (d=2)
% =========================================================
\begingroup
\small
\setlength{\tabcolsep}{4pt}
\setlength{\LTleft}{-20cm plus 1fill}
\setlength{\LTright}{-20cm plus 1fill}
\renewcommand{\arraystretch}{1.8}
\begin{longtable}{ >{$}c<{$} | >{$}c<{$} | >{$}c<{$} | >{$}c<{$} }
	\caption[Classification of finite-dimensional Nichols algebras for dimension 2]{Classification of finite-dimensional Nichols algebras for $\dim_{\protect\k} V = 2$}
	\label{tab:dim2_classification} \\
	\hline
	\makebox[2.6cm]{$n,\gamma$} & \makebox[2.8cm]{$\b, i$} & \makebox[5.4cm]{$\a, w, r$} & \makebox[1.6cm]{$(k,l)$} \\
	\hline\hline
	\endfirsthead
	
	\caption[]{Classification of finite-dimensional Nichols algebras for $\dim_{\protect\k} V = 2$ (continued)} \\
	\hline
	n,\gamma & \b, i & \a, w, r & (k,l) \\
	\hline\hline
	\endhead
	
	\hline
	\endfoot
	
	\hline
	\endlastfoot
	
	% H_{2,1}
	n \geq 2 & \b = \gamma^{1-i} & \a^w = 1, \quad \a^r = \gamma^{i^2 - i + 1} & (2, 1) \\
	\hline
	% H_{3,1} Case 1
	n \geq 3 & \b = \gamma^{1-i} & \a^w = 1, \quad \a^r = -\gamma^{i^2 - i} & (3, 1) \\
	\hline
	% H_{3,1} Case 2
	\gamma = -1 & \b \notin \{-1, 1\} & \a^w = \b^2, \quad \a^r = (-1)^i \b^{1-i} & (3, 1) \\
	\hline
	% H_{3,2}
	\gamma = -1 & \b \notin \{-1, 1\} & \a^w = \b^2, \quad \a^r = -\b^{-i} & (3, 2) \\
	\hline
	% H_{4,1} Case 2
	n \geq 2 & \b = \gamma^{1-i} & \a^w = 1, \quad \a^{2r} = \gamma^{2i^2 - 2i + 1} & (4, 1) \\
	\hline
	% H_{5,1} Case 2
	\gamma = -1 & \b \notin \{-1, 1\} & \a^w = \b^2, \quad \a^{2r} = (-1)^i \b^{1-2i} & (5, 1),(5,2) \\
	\hline
	% H_{6,1} Case 2 
	n \geq 2, \ n \neq 3 & \b = \gamma^{1-i} & \a^w = 1, \quad \a^r \gamma^{i-i^2} \in R_3 & (6, 1),(6,2) \\
	\hline
	% H_{7,1} Case 2
	\gamma = -1 & (-1)^{i+1} \b \in R_3 & \a^w = \b^2, \quad \a^r = (-1)^{i+1} \b^{-1-i} & (7, 1), (7, 2) \\
	\hline
	% H_{8,2} Case 2
	\gamma = -1 & \b \in R_{12} & \a^w = \b^2, \quad \a^r = \b^{4-i} & (8, 2),(8,3) \\
	\hline
	% H_{8,4} Case 2
	\gamma = -1 & \b \in R_{12} & \a^w = \b^2, \quad \a^r = (-1)^{i} \b^{3-i} & (8, 4), (8,5) \\
	\hline
	% H_{9,2} Case 2
	\gamma = -1 & \b \in R_4 & \a^w = -1, \quad \a^r \b^i \in R_3 & (9, 2) \\
	\hline
	% H_{9,3} Case 2
	\gamma = -1 & \begin{array}{@{}c@{}} \b = (-1)^{i+1}\zeta^{-3} \\ (\text{where } \zeta \in R_{12}) \end{array} & \a^w = -1, \quad \a^r = -\zeta^{3i-1} & (9, 3) \\
	\hline
	% H_{10,2} Case 2
	\gamma = -1 & (-1)^i \b \in R_9 & \a^w = \b^2, \quad \a^r = (-1)^i \b^{3-i} & (10, 2) \\
	\hline
	% H_{10,3} Case 2
	\gamma = -1 & (-1)^i \b \in R_9 & \a^w = \b^2, \quad \a^r = -\b^{-i-2} & (10, 3) \\
	\hline
	% H_{11,1} Case 2 
	n \geq 2 & \b = \gamma^{1-i} & \a^w = 1, \quad \a^{3r} = \gamma^{3i^2 - 3i + 1} & (11, 1) \\
	\hline
	% H_{12,1} Case 2
	n = 8 & \b = \gamma^{1-i} & \a^w = 1, \quad \a^r = \gamma^{i^2-i-2} & (12, 1) \\
	\hline
	% H_{12,2} Case 2
	\gamma = -1 & \b \in R_8 & \a^w = \b^2, \quad \a^r = \b^{2-i} & (12, 2) \\
	\hline
	% H_{12,3} Case 2
	\gamma = -1 & \b \in R_8 & \a^w = \b^2, \quad \a^r = (-1)^{i} \b^{3-i} & (12, 3) \\
	\hline
	% H_{13,2} Case 2
	n = 24 & \b = \gamma^{1-i} & \a^w = 1, \quad \a^r = \gamma^{i^2-i-6} & (13, 2) \\
	\hline
	% H_{13,3} Case 2
	\gamma = -1 & \b \in R_{24} & \a^w = \b^2, \quad \a^r = \b^{8-i} & (13, 3) \\
	\hline
	% H_{13,4} Case 2
	\gamma = -1 & \b \in R_{24} & \a^w = \b^2, \quad \a^r = (-1)^i \b^{5-i} & (13, 4) \\
	\hline
	% H_{14,1} Case 2
	\gamma = -1 & (-1)^i \b \in R_5 & \a^w = \b^2, \quad \a^r = (-1)^i \b^{-i-3} & (14, 1) \\
	\hline
	% H_{14,2} Case 2
	\gamma = -1 & (-1)^i \b \in R_5 & \a^w = \b^2, \quad \a^r = (-1)^{i+1} \b^{-i-1} & (14, 2) \\
	\hline
	% H_{15,1} Case 2
	\gamma = -1 & \b \in R_{20} & \a^w = \b^2, \quad \a^r = (-1)^i \b^{7-i} & (15, 1),(15,2) \\
	\hline
	% H_{15,3} Case 2
	\gamma = -1 & \b \in R_{20} & \a^w = \b^2, \quad \a^r = \b^{4-i} & (15, 3),(15,4) \\
	\hline
	% H_{16,2} Case 2
	n = 30 & \b = \gamma^{1-i} & \a^w = 1, \quad \a^r = \gamma^{i^2-i-12} & (16, 2) \\
	\hline
	% H_{16,3} Case 2
	\gamma = -1 & (-1)^i \b \in R_{30} & \a^w = \b^2, \quad \a^r = \b^{10-i} & (16, 3) \\
	\hline
	% H_{16,4} Case 2
	\gamma = -1 & (-1)^i \b \in R_{30} & \a^w = \b^2, \quad \a^r = \b^{6-i} & (16, 4) \\
	\hline
	% H_{17,1} Case 2
	\gamma = -1 & (-1)^i \b \in R_{14} & \a^w = \b^2, \quad \a^r = (-1)^i \b^{5-i} & (17, 1) \\
	\hline
	% H_{17,2} Case 2
	\gamma = -1 & (-1)^i \b \in R_{14} & \a^w = \b^2, \quad \a^r = -\b^{10-i} & (17, 2) \\
\end{longtable}
\endgroup

% =========================================================
% Table 3: 对应 YD 模维数为 3 的情况 (d=3)
% =========================================================
\begingroup
\small
\setlength{\tabcolsep}{4pt}
\setlength{\LTleft}{-20cm plus 1fill}
\setlength{\LTright}{-20cm plus 1fill}
\renewcommand{\arraystretch}{1.8}
\begin{longtable}{ >{$}c<{$} | >{$}c<{$} | >{$}c<{$} | >{$}c<{$} }
	\caption[Classification of finite-dimensional Nichols algebras for dimension 3]{Classification of finite-dimensional Nichols algebras for $\dim_{\protect\k} V = 3$}
	\label{tab:dim3_classification} \\
	\hline
	\makebox[2.6cm]{$n,\gamma$} & \makebox[2.8cm]{$\b, i$} & \makebox[5.4cm]{$\a, w, r$} & \makebox[1.6cm]{$(k,l)$} \\
	\hline\hline
	\endfirsthead
	
	\caption[]{Classification of finite-dimensional Nichols algebras for $\dim_{\protect\k} V = 3$ (continued)} \\
	\hline
	n,\gamma & \b, i & \a, w, r & (k,l) \\
	\hline\hline
	\endhead
	
	\hline
	\endfoot
	
	\hline
	\endlastfoot
	
	% H_{4,1} Case 1
	n \geq 3 & \b = \gamma^{2-i} & \a^w = 1, \quad \a^r = \gamma^{i^2 - 2i + 2} & (4, 1) \\
	\hline
	% H_{5,1} Case 1
	n \geq 3, \ n \neq 4 & \b = \gamma^{2-i} & \a^w = 1, \quad \a^r = -\gamma^{i^2 - 2i} & (5, 1),(5,2) \\
	\hline
	% H_{6,1} Case 1 
	n = 3 & \b^3 \neq 1 & \a^w = \b^3, \quad \a^r = \b^{1-i} \gamma^i & (6, 1),(6,2) \\
	\hline
	% H_{7,1} Case 1
	n = 3 & \b = -\gamma^{-1-i} & \a^w = -1, \quad \a^r = (-1)^{i+1} \gamma^{i^2+i} & (7, 1), (7, 2) \\
	\hline
	% H_{8,1}  
	n = 3 & \b\gamma^{i} \in R_4 & \a^w = \b^3, \quad \a^r = \b^{-i} \gamma^{-1} & (8, 1) \\
	\hline
	% H_{8,2} Case 1
	\begin{array}{@{}c@{}} n = 3, \ \gamma = -\zeta^{-2} \\ (\text{where } \zeta \in R_{12}) \end{array} & \b = (-1)^i \zeta^{2i+1} & \a^w = \zeta^3, \quad \a^r = (-1)^{i+1} \zeta^{-2i^2-i} & (8, 2),(8,3) \\
	\hline
	% H_{9} series
	\begin{array}{@{}c@{}} n = 3, \ \gamma = -\zeta^2 \\ (\text{where } \zeta \in R_{12}) \end{array} & \b = (-1)^i \zeta^{-2i-1} & \a^w = \zeta^{-3}, \quad \a^r = (-1)^{i+1} \zeta^{2i^2+i+2} & (9, 1) \\
	\hline
	% H_{9,2} Case 1
	n = 3 & \b\gamma^{i} \in R_4 & \a^w = \b^3, \quad \a^r = -\b^{-i} & (9, 2) \\
	\hline
	% H_{10} series Case 1
	n = 18 & \b = \gamma^{2-i} & \a^w = 1, \quad \a^r = \gamma^{i^2-2i+12} & (10, 1) \\
	\hline
	% H_{10} series Case 2
	\begin{array}{@{}c@{}} n = 3, \ \gamma = \zeta^3 \\ (\text{where } \zeta \in R_9) \end{array} & \b = \zeta^{2-3i} & \a^w = \zeta^6, \quad \a^r = -\zeta^{3i^2-2i+1} & (10, 1) \\
	\hline
	% H_{10,2} Case 1
	\begin{array}{@{}c@{}} n = 3, \ \gamma = \zeta^3 \\ (\text{where } \zeta \in R_9) \end{array} & \b = \zeta^{-3i+1} & \a^w = \zeta^3, \quad \a^r = -\zeta^{3i^2-i} & (10, 2) \\
	\hline
	% H_{13,1} Case 2
	\begin{array}{@{}c@{}} n = 3, \ \gamma = -\zeta^{-4} \\ (\text{where } \zeta \in R_{24}) \end{array} & \b = (-1)^{i+1} \zeta^{4i+1} & \a^w = -\zeta^3, \quad \a^r = \zeta^{-4i^2-i+6} & (13, 1) \\
	\hline
	% H_{13,3} Case 1
	\begin{array}{@{}c@{}} n = 3, \ \gamma = -\zeta^{-4} \\ (\text{where } \zeta \in R_{24}) \end{array} & \b = (-1)^i \zeta^{4i-5} & \a^w = \zeta^9, \quad \a^r = (-1)^{i+1} \zeta^{-4i^2+5i} & (13, 3) \\
	\hline
	% H_{16,1} Case 2
	n=3 & \b \gamma^{i} \in R_{10} & \a^w = \b^3, \quad \a^r = \b^{-3-i} \gamma^{2-3i} & (16, 1) \\
	\hline
	% H_{16,3} Case 1
	\begin{array}{@{}c@{}} n = 3, \ \gamma = \zeta^5 \\ (\text{where } \zeta \in R_{15}) \end{array} & \b = -\zeta^{-5i+2} & \a^w = -\zeta^6, \quad \a^r = (-1)^{i+1} \zeta^{5i^2-2i} & (16, 3) \\
\end{longtable}
\endgroup

% =========================================================
% Table 4: 对应 YD 模维数为 4 的情况 (d=4)
% =========================================================
\begingroup
\small
\setlength{\tabcolsep}{4pt}
\setlength{\LTleft}{-20cm plus 1fill}
\setlength{\LTright}{-20cm plus 1fill}
\renewcommand{\arraystretch}{1.8}
\begin{longtable}{ >{$}c<{$} | >{$}c<{$} | >{$}c<{$} | >{$}c<{$} }
	\caption[Classification of finite-dimensional Nichols algebras for dimension 4]{Classification of finite-dimensional Nichols algebras for $\dim_{\protect\k} V = 4$}
	\label{tab:dim4_classification} \\
	\hline
	\makebox[2.6cm]{$n,\gamma$} & \makebox[2.8cm]{$\b, i$} & \makebox[5.4cm]{$\a, w, r$} & \makebox[1.6cm]{$(k,l)$} \\
	\hline\hline
	\endfirsthead
	
	\caption[]{Classification of finite-dimensional Nichols algebras for $\dim_{\protect\k} V = 4$ (continued)} \\
	\hline
	n,\gamma & \b, i & \a, w, r & (k,l) \\
	\hline\hline
	\endhead
	
	\hline
	\endfoot
	
	\hline
	\endlastfoot
	
	% H_{8,4} Case 1
	\begin{array}{@{}c@{}} n = 4, \ \gamma = -\zeta^3 \\ (\text{where } \zeta \in R_{12}) \end{array} & \b = (-1)^i \zeta^{-3i-1} & \a^w = \zeta^8, \quad \a^r = (-1)^{i+1} \zeta^{3i^2+i} & (8, 4),(8,5) \\
	\hline
	% H_{9,3} Case 1
	n = 12 & \b = \gamma^{3-i} & \a^w = 1, \quad \a^r = -\gamma^{i^2-3i} & (9, 3) \\
	\hline
	% H_{11,1} Case 1
	n \geq 4 & \b = \gamma^{3-i} & \a^w = 1, \quad \a^r = \gamma^{i^2-3i+3} & (11, 1) \\
	\hline
	% H_{12,1} Case 1
	\begin{array}{@{}c@{}} n = 4, \ \gamma = \zeta^2 \\ (\text{where } \zeta \in R_8) \end{array} & \b = \zeta^{-2i-1} & \a^w = -1, \quad \a^r = \zeta^{2i^2+i-1} & (12, 1) \\
	\hline
	% H_{12,2} Case 1
	\begin{array}{@{}c@{}} n = 4, \ \gamma = \zeta^2 \\ (\text{where } \zeta \in R_8) \end{array} & \b = -\zeta^{-2i+1} & \a^w = -1, \quad \a^r = (-1)^{i+1} \zeta^{2i^2-i} & (12, 2) \\
	\hline
	% H_{12,3} Case 1
	n = 8 & \b = \gamma^{3-i} & \a^w = 1, \quad \a^r = -\gamma^{i^2-3i} & (12, 3) \\
	\hline
	% H_{13,1} Case 1
	\begin{array}{@{}c@{}} n = 4, \ \gamma = \zeta^6 \\ (\text{where } \zeta \in R_{24}) \end{array} & \b = -\zeta^{-6i+1} & \a^w = \zeta^4, \quad \a^r = (-1)^{i+1} \zeta^{6i^2-i-4} & (13, 1) \\
	\hline
	% H_{13,2} Case 1
	\begin{array}{@{}c@{}} n = 4, \ \gamma = \zeta^6 \\ (\text{where } \zeta \in R_{24}) \end{array} & \b = \zeta^{-6i-1} & \a^w = \zeta^{-4}, \quad \a^r = \zeta^{6i^2+i-1} & (13, 2) \\
	\hline
	% H_{14,1} Case 1
	n = 5 & \b = \gamma^{3-i} & \a^w = 1, \quad \a^r = -\gamma^{i^2+2i} & (14, 1) \\
	\hline
	% H_{15,1} Case 1
	n = 20 & \b = \gamma^{3-i} & \a^w = 1, \quad \a^r = -\gamma^{i^2-3i} & (15, 1),(15,2) \\
	\hline
	% H_{16,1} Case 1
	n=30 & \b = \gamma^{3-i} & \a^w = 1, \quad \a^r = -\gamma^{i^2-3i+5} & (16, 1) \\
	\hline
	% H_{17,1} Case 1
	n = 14 & \b = \gamma^{3-i} & \a^w = 1, \quad \a^r = -\gamma^{i^2-3i} & (17, 1) \\
\end{longtable}
\endgroup

% =========================================================
% Table 5: 对应 YD 模维数为 5 的情况 (d=5)
% =========================================================
\begingroup
\small
\setlength{\tabcolsep}{4pt}
\setlength{\LTleft}{-20cm plus 1fill}
\setlength{\LTright}{-20cm plus 1fill}
\renewcommand{\arraystretch}{1.8}
\begin{longtable}{ >{$}c<{$} | >{$}c<{$} | >{$}c<{$} | >{$}c<{$} }
	\caption[Classification of finite-dimensional Nichols algebras for dimension 5]{Classification of finite-dimensional Nichols algebras for $\dim_{\protect\k} V = 5$}
	\label{tab:dim5_classification} \\
	\hline
	\makebox[2.6cm]{$n,\gamma$} & \makebox[2.8cm]{$\b, i$} & \makebox[5.4cm]{$\a, w, r$} & \makebox[1.6cm]{$(k,l)$} \\
	\hline\hline
	\endfirsthead
	
	\caption[]{Classification of finite-dimensional Nichols algebras for $\dim_{\protect\k} V = 5$ (continued)} \\
	\hline
	n,\gamma & \b, i & \a, w, r & (k,l) \\
	\hline\hline
	\endhead
	
	\hline
	\endfoot
	
	\hline
	\endlastfoot
	
	% H_{10,3} Case 1
	n = 18 & \b = \gamma^{4-i} & \a^w = 1, \quad \a^r = -\gamma^{i^2-4i} & (10, 3) \\
	\hline
	% H_{14,2} Case 1
	n = 10 & \b = \gamma^{4-i} & \a^w = 1, \quad \a^r = -\gamma^{i^2-4i} & (14, 2) \\
	\hline
	% H_{15,3} Case 1
	\begin{array}{@{}c@{}} n = 5, \ \gamma = -\zeta^{-2} \\ (\text{where } \zeta \in R_{20}) \end{array} & \b = (-1)^i \zeta^{2i-3} & \a^w = \zeta^5, \quad \a^r = (-1)^{i+1} \zeta^{-2i^2+3i} & (15, 3),(15,4) \\
	\hline
	% H_{16,2} Case 1
	\begin{array}{@{}c@{}} n = 5, \ \gamma = \zeta^3 \\ (\text{where } \zeta \in R_{15}) \end{array} & \b = -\zeta^{-3i-4} & \a^w = -\zeta^{-5}, \quad \a^r = (-1)^{i+1} \zeta^{3i^2+4i-4} & (16, 2) \\
	\hline
	% H_{16,4} Case 1
	\begin{array}{@{}c@{}} n = 5, \ \gamma = \zeta^3 \\ (\text{where } \zeta \in R_{15}) \end{array} & \b = -\zeta^{-3i-2} & \a^w = -\zeta^5, \quad \a^r = (-1)^{i+1} \zeta^{3i^2+2i} & (16, 4) \\
\end{longtable}
\endgroup

% =========================================================
% Table 6: 对应 YD 模维数为 6 的情况 (d=6)
% =========================================================
\begingroup
\small
\setlength{\tabcolsep}{4pt}
\setlength{\LTleft}{-20cm plus 1fill}
\setlength{\LTright}{-20cm plus 1fill}
\renewcommand{\arraystretch}{1.8}
\begin{longtable}{ >{$}c<{$} | >{$}c<{$} | >{$}c<{$} | >{$}c<{$} }
	\caption[Classification of finite-dimensional Nichols algebras for dimension 6]{Classification of finite-dimensional Nichols algebras for $\dim_{\protect\k} V = 6$}
	\label{tab:dim6_classification} \\
	\hline
	\makebox[2.6cm]{$n,\gamma$} & \makebox[2.8cm]{$\b, i$} & \makebox[5.4cm]{$\a, w, r$} & \makebox[1.6cm]{$(k,l)$} \\
	\hline\hline
	\endfirsthead
	
	\caption[]{Classification of finite-dimensional Nichols algebras for $\dim_{\protect\k} V = 6$ (continued)} \\
	\hline
	n,\gamma & \b, i & \a, w, r & (k,l) \\
	\hline\hline
	\endhead
	
	\hline
	\endfoot
	
	\hline
	\endlastfoot
	
	% H_{13,4} Case 1
	n = 24 & \b = \gamma^{5-i} & \a^w = 1, \quad \a^r = - \gamma^{i^2-5i} & (13, 4) \\
	\hline
	% H_{17,2} Case 1
	n = 14 & \b = \gamma^{5-i} & \a^w = 1, \quad \a^r = -\gamma^{i^2-5i} & (17, 2) \\
\end{longtable}
\endgroup

\begin{remark}
\normalfont	
We make the following observations regarding our classification:
\begin{itemize}
	\item [(1)] In Table \ref{tab:dim1_classification}, the last column of the first row is empty. This corresponds to the case $n=1$, where the Yetter-Drinfeld module $V$ is necessarily one-dimensional. As discussed at the beginning of this subsection, we determine whether $\B(V)$ is finite-dimensional directly from its braiding, instead of using Lemma \ref{lem:WandD} (which requires $n\geq 2$). Therefore, no Heckenberg diagram is assigned to this case.
	\item [(2)] Our classification of finite-dimensional Nichols algebras yields significantly more cases than those obtained for $D_{i,j}$ in \cite[Section 5.2]{ZLY25}. This is because, in our current setting, the parameter $\b$ is not required to be of the form $\gamma^j$, and we include an additional parameter $\a$.
\end{itemize}

\end{remark}

\subsection{The case $(\a, \b, x^r g^i) \in \mathcal{T}_2$}

We now turn to the second case, where $(\a, \b, x^r g^i) \in \mathcal{T}_2$. Throughout this subsection, we continue to fix the Hopf algebra $H = B(n, w, \gamma)$ and consider the Yetter-Drinfeld module $V$ over it. We will prove that the associated Nichols algebra $\B(V(\a, \b, x^r g^i))$ is infinite-dimensional, regardless of the specific choices for the parameters $n, w$, and $\gamma$. While our approach strictly follows the methodology established in \cite[Section 5.3]{ZLY25}, we present the full proof here for the sake of completeness.

Assume $(\a, \b, x^r g^i) \in \mathcal{T}_2$ and set $ V = V(\a, \b, x^r g^i) $. Thanks to Theorem \ref{thm:YDmodule}, we can just assume $ 0\leq i \leq n-1 $. Recall the basis $\{v_0, v_1, \dots, v_{n-1}\}$ of $V$ from~\eqref{basisofV}. For each $ 0\leq k \leq n-1 $, let $ V_k = \k \{v_k \} $ be a $1$-dimensional space. We define the projection map as:
\[
p_k : V \to V_k, \quad \sum_{l=0}^{n-1} a_l v_l \mapsto a_k v_k.
\]
For brevity, we omit the tensor symbol. We write $ u_1 \dotsm u_m $ instead of $ u_1 \otimes \cdots \otimes u_m $. We want to prove that for all $ m \geq 1 $,
\begin{equation}
	\Delta_{1^m}^{T(V)}(v_i^m) \neq 0.
\end{equation}
If this holds, then $ \B(V) $ is infinite-dimensional.

For $ m\geq 1 $, define $ \psi_m =p_i^{\otimes m}\circ  \Delta_{1^m} $. We only need to show that 
\begin{equation}\label{psinot0}
	\psi_m(v_i^m) \neq 0.
\end{equation}
Note that $\psi_m(v_i^m) = a_m  v_i \otimes \cdots \otimes v_i$ for some scalar $a_m \in \k$. So, we just need to prove $a_m \neq 0$. To do this, we calculate $\Delta_{1^m}^{T(V)}(v_i^m)$ step by step.

For $ m\geq 2 $, we have the formula:
\begin{equation}\label{formuladelta}
	\Delta_{1^m}^{T(V)}(v_i^m)=(\pi_1 \otimes\Delta_{1^{m-1}}^{T(V)} ) ( \Delta (v_i^m) ).
\end{equation}
Since $\Delta$ is an algebra map, we get:
\begin{equation*}
	\Delta (v_i^m) = (1\otimes v_i + v_i \otimes 1)^m = \sum x_1 x_2 \cdots x_m.
\end{equation*}
Here, each $x_k$ is either $1 \otimes v_i$ or $v_i \otimes 1$. If a term has exactly $l$ copies of $v_i \otimes 1$ and $m-l$ copies of $1 \otimes v_i$, it belongs to $V^{\otimes l} \otimes V^{\otimes (m-l)}$.

In equation \eqref{formuladelta}, the map $\pi_1$ forces us to look only at terms in $V \otimes V^{\otimes (m-1)}$. These are the terms with exactly one $v_i \otimes 1$. Thus, we find:
\begin{equation}\label{formuladelta11}
	\Delta_{1^m}^{T(V)}(v_i^m)=(\pi_1 \otimes\Delta_{1^{m-1}}^{T(V)} ) \left(  \sum_{k=1}^{m} (1\otimes v_i^{k-1})(v_i \otimes v_i^{m-k}) \right).
\end{equation}

Now, consider the term $(1\otimes v_i^{k-1})(v_i \otimes v_i^{m-k})$. If $k = 1$, this is just $v_i \otimes v_i^{m-1}$. If $k > 1$, we expand it as:
\begin{equation*}
	(1\otimes v_i^{k-1})(v_i \otimes v_i^{m-k})=(v_i^{k-1})_{-1} \cdot v_i \otimes (v_i^{k-1})_{0}v_i^{m-k}=\sum a_{j_1, \dots, j_m} v_{j_1}\otimes v_{j_2}\cdots v_{j_m},
\end{equation*}
where $a_{j_1, \dots, j_m} \in \k$ and $0 \leq j_l \leq n-1$. Based on the action and coaction on $V$, for any term with a non-zero coefficient (meaning $a_{j_1, \dots, j_m} \neq 0$), the indices must satisfy:
\begin{equation*}
	j_1 + \cdots + j_m \leq m i.
\end{equation*}
In fact, a more general inequality holds, which we state as the following lemma.

\begin{lemma}\label{lem:zhibiaobuzeng}
For any $m \geq 2$, consider the expansion:
\[
\Delta_{1^m}^{T(V)}(v_{i_1}\otimes \cdots \otimes v_{i_m})=\sum a_{j_1, \dots, j_m} v_{j_1}\otimes \cdots \otimes v_{j_m},
\]
where $a_{j_1, \dots, j_m} \in \k$ and $0 \leq i_l, j_l \leq n-1$. Whenever $a_{j_1, \dots, j_m} \neq 0$, the indices satisfy:
\[
j_1 + \cdots + j_m \leq i_1 + \cdots + i_m.
\]
\end{lemma}

We now prove \eqref{psinot0}. Using the facts above, we compute $\psi_m(v_i^m)$:
\begin{align*}
	\psi_m(v_i^m) 
	&=  (p_i^{\otimes m}\circ  \Delta_{1^m}) (v_i^m) \\
	&= \left(p_i \otimes p_i^{\otimes^{m-1}}   \right) (\pi_1 \otimes\Delta_{1^{m-1}}^{T(V)} ) \left(  \sum_{k=1}^{m} (1\otimes v_i^{k-1})(v_i \otimes v_i^{m-k}) \right) \\
	&= \sum_{k=1}^{m} \left(p_i  \pi_1 \otimes \psi_{m-1} \right) \left( (v_i^{k-1})_{-1} \cdot v_i \otimes (v_i^{k-1})_{0}v_i^{m-k}  \right).  
\end{align*}
For each $1 \leq k \leq m$, we claim:
\begin{equation*}
	\left(p_i  \pi_1 \otimes \psi_{m-1} \right) \left( (v_i^{k-1})_{-1} \cdot v_i \otimes (v_i^{k-1})_{0}v_i^{m-k}  \right)=(\a^r)^{k-1}v_i\otimes  \psi_{m-1}(v_i^{m-1}).
\end{equation*}
The case $k = 1$ is clear. For $k \geq 2$, the steps are similar. We show the case $k = 3$ as an example. Since $c_\b^{r,i}(i,i) = x^r$, we get:
\begin{align*}
	& \left(p_i  \pi_1 \otimes \psi_{m-1} \right) \left( (v_i^{2})_{-1} \cdot v_i \otimes (v_i^{2})_{0}v_i^{m-3}  \right) \\
	&=  \left(p_i  \pi_1 \otimes \psi_{m-1} \right) \left( \sum_{0\leq r,l \leq i} c_\b^{r,i}(i,r)c_\b^{r,i}(i,l) \cdot v_i \otimes v_rv_l	v_i^{m-3} \right) \\
	&= (\a^r)^2 v_i\otimes  \psi_{m-1}(v_i^{m-1}).
\end{align*}
This last step follows from Lemma \ref{lem:zhibiaobuzeng}, which forces $\psi_{m-1}(v_r v_l v_i^{m-3}) = 0$ whenever $r < i$ or $l < i$. Substituting this back into the main sum, we obtain:
\begin{align*}
	\psi_m(v_i^m) 
	&=  \sum_{k=1}^{m} \left(p_i  \pi_1 \otimes \psi_{m-1} \right) \left( (v_i^{k-1})_{-1} \cdot v_i \otimes (v_i^{k-1})_{0}v_i^{m-k}  \right) \\
	&= \left(\sum_{k=1}^m (\a^r)^{k-1} \right) \left(v_i\otimes  \psi_{m-1}(v_i^{m-1}) \right) \\
	&= \left(\sum_{k=1}^m (\a^r)^{k-1} \right)\left(\sum_{k=1}^{m-1} (\a^r)^{k-1} \right) \left(v_i\otimes  v_i  \otimes \psi_{m-2}(v_i^{m-2}) \right) \\
	&= \left(\prod_{l=1}^{m} \left(\sum_{k=1}^l (\a^r)^{k-1} \right)    \right)( v_i \otimes \cdots \otimes v_i).
\end{align*}
Since $(\a, \b, x^r g^i) \in \mathcal{T}_2$, we know either $\a^r = 1$ or $\a^r \notin R_{\infty}$. In both cases, the product is not zero:
\[
\prod_{l=1}^{m} \biggl( \sum_{k=1}^l (\a^r)^{k-1} \biggr) \neq 0.
\]
Therefore, $\psi_m(v_i^m) \neq 0$ for all $m \geq 1$. This leads to our final result.

\begin{lemma}\label{lem:finiten-nichols}
	The Nichols algebra $\B(V(\a, \b, x^r g^i))$ is infinite-dimensional for all $(\a, \b, x^r g^i) \in \mathcal{T}_2$.
\end{lemma}

We conclude this section with the following theorem.

\begin{theorem}\label{thm:nicholsalgebra}
Let $H = B(n, w, \gamma)$. Then the simple Yetter-Drinfeld modules $V$ over $H$ for which the Nichols algebra $\B(V)$ is finite-dimensional are precisely those classified in Tables \ref{tab:dim1_classification}--\ref{tab:dim6_classification} following Lemma \ref{lem:finiten-nichols-classification}.
\end{theorem}

\begin{proof}
This follows immediately from Lemmas~\ref{lem:finiten-nichols-classification},  and \ref{lem:finiten-nichols}.
\end{proof}

\section*{}

\subsection*{Funding}
Supported by National Key R\&D Program of China 2024YFA1013802 and NSFC 12271243.

\subsection*{Data availability}
No data was used for the research described in the article.

%\subsection*{Author Contributions}
%All authors contributed to all aspects of this project.
%\section*{Acknowledgements}
%The authors would like to thank the referee for his/her detailed and valuable comments and suggestions.

%\subsection*{Availability of data and materials}
%Not applicable.
%section*{Declarations}
%\subsection*{Competing interests}
%The authors declare no competing interests.
%\subsection*{Ethical Approval} Not applicable.
%The authors declare that they have no known competing financial interests or personal relationships that could have appeared to influence the work reported in this paper.
%\section*{Conflict of interest}
%Declarations of interest: none.

\end{document}